\documentclass[a4paper]{article}

\usepackage{amsmath,amsthm,amssymb}
\usepackage{hyperref}
\usepackage[disable]{todonotes}
\usepackage[all]{xy}

\newtheorem{theorem}{Theorem}[section]
\newtheorem{proposition}[theorem]{Proposition}
\newtheorem{lemma}[theorem]{Lemma}
\newtheorem{corollary}[theorem]{Corollary}
\newtheorem{remark}[theorem]{Remark}
\newtheorem{conjecture}[theorem]{Conjecture}
\theoremstyle{definition}

\newtheorem{definition}[theorem]{Definition}

\newcommand{\pushoutcorner}[1][dr]{\save*!/#1+1.2pc/#1:(1,-1)@^{|-}\restore}
\newcommand{\pullbackcorner}[1][dr]{\save*!/#1-1.2pc/#1:(-1,1)@^{|-}\restore}

\newcommand{\cat}[1]{\mathbb{#1}}
\newcommand{\catc}{\cat{C}}
\newcommand{\set}{\mathbf{Set}}
\newcommand{\asm}{\mathbf{Asm}}

\newcommand{\btwo}{\mathbf{2}}

\newcommand{\smcat}[1]{\mathcal{#1}}
\newcommand{\cod}{\operatorname{cod}}

\newcommand{\intv}{\mathbb{I}}

\newcommand{\names}{\mathbb{A}}

\newcommand{\cons}{\mathtt{cons}}

\newcommand{\concat}{\mathtt{trans}}
\newcommand{\mrgl}{\mathtt{reduceleft}}
\newcommand{\mrgr}{\mathtt{reduceright}}
\newcommand{\ext}{\mathtt{extn}}

\newcommand{\nat}{\mathbb{N}}

\newcommand{\realno}{\mathbb{R}}

\newcommand{\klone}{\mathcal{K}_1}
\newcommand{\kltwo}{\mathcal{K}_2}

\newcommand{\czf}{\mathbf{CZF}}

\newcommand{\wisc}{\mathbf{WISC}}

\newcommand{\yoneda}{\mathbf{y}}
\newcommand{\im}{\operatorname{im}}

\newcommand{\eff}{\mathcal{E}ff}
\newcommand{\kv}{\mathcal{KV}}
\newcommand{\rt}{\mathbf{RT}}
\newcommand{\gn}[1]{\ulcorner #1 \urcorner}

\newcommand{\dm}{\mathbf{dM}}

\title{$W$-Types with Reductions and the Small Object Argument}
\author{Andrew W Swan}

\begin{document}

\maketitle

\begin{abstract}
  We define a simple kind of higher inductive type generalising
  dependent $W$-types, which we refer to as \emph{$W$-types with
    reductions}. Just as dependent $W$-types can be characterised as
  initial algebras of certain endofunctors (referred to as
  \emph{polynomial endofunctors}), we will define our generalisation
  as initial algebras of certain pointed endofunctors, which we will
  refer to as \emph{pointed polynomial endofunctors}.

  We will show that $W$-types with reductions exist in all $\Pi
  W$-pretoposes that satisfy a weak choice axiom, known as
  \emph{weakly initial set of covers} ($\wisc$). This includes all
  Grothendieck toposes and realizability toposes as long as $\wisc$ holds
  in the background universe.

  We will show that a large class of $W$-types with reductions in
  internal presheaf categories can be constructed without using
  $\wisc$.

  We will show that $W$-types with reductions suffice to construct some
  interesting examples of algebraic weak factorisation
  systems (awfs's). Specifically, we will see how to construct awfs's
  that are cofibrantly generated with respect to a codomain
  fibration, as defined in a previous paper by the author.
\end{abstract}

\section{Introduction}
\label{sec:introduction}

A key idea in type theory is that of inductively generated types. The
essential idea is that one specifies a way to construct new elements
of a type from old, and an inductively generated type is the ``least''
type matching this specification. The simplest example is the natural
numbers, $\nat$. It is the type inductively generated by the
requirements that $0$ is an element of $\nat$ and $S(n)$ is an
element of $\nat$ whenever $n$ is. Since $\nat$ is the least such
type, we can prove a formula $\varphi$ holds for all natural numbers $n$,
by first proving $\varphi$ for $0$, then showing $\varphi$ holds for
$S(n)$ whenever it holds for $n$.

An important class of inductive types is that of $W$-types. These have
elegant categorical semantics due to Moerdijk and Palmgren
\cite{moerdijkpalmgrenwtypes}, and later developed further to
dependent $W$-types by Gambino and Hyland \cite{gambinohylanddepw}. In
these semantics, $W$-types are implemented as initial algebras of a
certain class of endofunctors, known as \emph{polynomial
  endofunctors}. Type theoretically the idea (for the simpler non
dependent case) is that we are given a
type $Y$ that we refer to as \emph{constructors} and a family of types
$X_y$ indexed by the elements of $y$, which we refer to as
\emph{arities}. We then construct a type $W$,
which contains an element of the form $\sup(y, \alpha)$ whenever $y
\in Y$ and $\alpha \colon X_y \rightarrow W$.

Higher inductive types are one of the main ideas in homotopy type
theory \cite{hottbook}, in which one defines a new type by specifying
not only how to construct elements of a type, but also how to
construct proofs of equality between elements (and also proofs of
equality between proofs of equality, etc). A lot of the time the aim
here is to construct types with nontrivial higher type structure that
represent interesting topological spaces (such as $n$ dimensional
spheres) type theoretically. However, there are examples of higher
inductive types that are non trivial even when working in an
extensional setting, where
UIP holds (any two proofs of equality are equal).
Many years before the term ``higher inductive
type'' was even coined, it was known that free algebras can be
constructed for (infinitary) varieties, and as observed by Blass, this
can even be carried out internally in a topos with a natural numbers
object satisfying the internal axiom of choice \cite[Section
8]{blassfreealg}. As observed by Lumsdaine and Shulman in the
introduction to
\cite{lumsdaineshulmanhit}, this can now be viewed as a
kind of higher inductive type.
More recently, in \cite{acdfqiit}
Altenkirch, Capriotti, Dijkstra and Forsberg developed a class of higher
inductive types, which they call \emph{quotient inductive-inductive
  types} which also
have interesting structure even within extensional type theory.
See also the
earlier work on \emph{quotient inductive types} by Altenkirch and
Kaposi in \cite{altenkirchkaposiqit}.

We will develop an idea for a simple kind of higher inductive type
that we will call $W$-type with reductions. Essentially, we identify
$\sup(y, \alpha)$ with some of the elements $\alpha(x)$ used to
construct it.

Although $W$-types with reductions are relatively simple, we will see
that they have an interesting application in homotopical algebra and
the semantics of homotopy type theory. A well known construction in
homotopical algebra is Garner's small object argument
\cite{garnersmallobject}, in which a cofibrantly generated algebraic
weak factorisation system (awfs) is constructed, making essential use
of transfinite colimits. In an earlier paper \cite{swanliftprob} the
author defined a new generalised definition of cofibrantly generated
within a Grothendieck fibration, and showed that to construct a
cofibrantly generated awfs in this new sense, it suffices to show that
certain pointed endofunctors have initial algebras. We will show that
when working over the codomain fibration for a locally cartesian
closed category, these initial algebras can be seen as $W$-types
with reductions. This will then be used to construct some interesting,
previously unknown examples of awfs's.

$W$-types with reductions may turn out to be
special cases of free algebras for varieties and/or QIITs,
and just like with those they are non trivial
even when working in extensional type theory. Indeed throughout this
paper we will be working with locally cartesian closed categories
which we think of as models for extensional type theory.
However, the relative simplicity of $W$-types with
reductions will have some important advantages. We will show how the
semantics for dependent $W$-types can be generalised to also give us
semantics for $W$-types with reductions. We will then show that
$W$-types with reductions can be implemented in any $\Pi W$-pretopos
satisfying a weak choice axiom known as $\wisc$ (such categories are
sometimes referred to as \emph{predicative toposes}
\cite{vdbergpredtop}).
An interesting
aspect of this is that currently approaches to the semantics of higher
inductive types such as the work of Lumsdaine and Shulman in
\cite{lumsdaineshulmanhit} use transfinite colimits for the
construction of the underlying objects. On the other hand, there are
interesting examples of predicative toposes based on realizability
that do not have infinite colimits, that we will see in section
\ref{sec:exampl-prev-unkn}.
The key is that we will construct the types within the
internal logic of the predicative topos using $W$-types.

The main focus of this paper is on semantics, in the same spirit
as Gambino and Hyland in \cite{gambinohylanddepw}. We will,
however give an intuitive explanation of what $W$-types with
reductions look like in the internal logic of a $\Pi W$-pretopos,
which will suggest what a syntax for $W$-types with reductions might
look like.

\subsection{On Internal Languages for Locally Cartesian Closed
  Categories}
\label{sec:note-use-internal}

Throughout this paper we will use type theoretic notation for objects
in a locally cartesian closed category, and type theory style
arguments for some of the proofs. Often, given a map $f \colon X
\rightarrow Y$ we will think of it as a family of types
indexed by $Y$, written as $X_y$ or $X(y)$. This is justified
by the well known paper by Seely \cite{seelylcctt}, although
strictly speaking, in order to really interpret extensional type
theory one needs the later work by Hofmann in \cite{hofmannlcc}.

One can also add disjoint coproducts, propositional truncation and
effective quotients to the type theory, as long as the locally
cartesian closed category possesses the appropriate structure. See
e.g. the work of Maietti in \cite{maiettimodular}.

Furthermore, as shown by Moerdijk and Palmgren $W$-types in type
theory correspond closely to the categorical definition that we will
use here. See \cite{moerdijkpalmgrenwtypes} for more details.

In \cite[Remark 5.9]{moerdijkpalmgrenwtypes} Moerdijk and Palmgren
point out a subtle issue to bear in mind when working with
$W$-types. If we are constructing a map from a $W$-type, $W$ to an
object $A$, then it is very straightforward to convert an
argument by recursion in type theory into a direct argument using the
initial algebra property of $W$. However, sometimes in proofs we want
to construct a predicate on $W$ by induction. In this case there is
not a straightforward way to interpret such arguments in an arbitrary
locally cartesian closed category. However, as Moerdijk and Palmgren
show in \cite{moerdijkpalmgrenast1}, such arguments can be interpreted
in the richer structure of a \emph{stratified pseudotopos}, and that
many natural examples of $\Pi W$-pretoposes possess this additional
structure. In this paper we will sometimes see such arguments, since
they are often the most natural and easy to understand
proofs. However, our results do apply to arbitrary locally cartesian
categories and we will also include brief explanations of how the
proofs can be adapted to work in general.

\section{$W$-Types with Reductions}
\label{sec:w-types-with}

\subsection{Definition}

We recall from \cite{gambinohylanddepw} that Gambino and Hyland
defined the following notions of polynomial, dependent polynomial
endofunctor and dependent $W$-type, which we will
generalise.
Throughout we assume that we are given a locally cartesian
closed and finitely cocomplete category $\catc$.

\begin{definition}[Gambino and Hyland]
  \label{def:depwtype}
  A \emph{polynomial} is a diagram of the following form.
  \begin{equation*}
    \begin{gathered}
      \xymatrix { & X \ar[r]^f \ar[dl]_h & Y \ar[dr]^g \\
        Z & & & Z}    
    \end{gathered}
  \end{equation*}

  A \emph{dependent polynomial endofunctor} is an endofunctor
  $\catc/Z \rightarrow \catc/Z$ of the form $\Sigma_g \Pi_f h^\ast$,
  where $g$, $f$ and $h$ are as above. We denote
  this endofunctor as $P_{f, g, h}$.

  A \emph{dependent $W$-type} is an initial object in the category of
  $P_{f, g, h}$-algebras for some dependent polynomial endofunctor
  $P_{f, g, h}$.
\end{definition}

We now give the new more general definition of polynomial with
reductions and pointed polynomial endofunctor with reductions.

\begin{definition}
  \label{def:polyred}
  Suppose we are given maps $f, g, h$ and $r$ as in the following
  diagram.
  \begin{equation}
    \label{eq:27}
    \begin{gathered}
      \xymatrix { R \ar[r]^k & X \ar[r]^f \ar[dl]_h & Y \ar[dr]^g \\
        Z & & & Z}    
    \end{gathered}
  \end{equation}
  We say the diagram is \emph{coherent}, or satisfies the
  \emph{coherence condition} if $g \circ f \circ k = h \circ k$.

  We say that a diagram as in \eqref{eq:27} satisfying the coherence
  condition is a \emph{polynomial with reductions}.

  We refer to the subdiagram consisting of $f$, $g$ and $h$ as the
  \emph{underlying polynomial}, and to $R$ and $k$ as the
  \emph{reductions}.
\end{definition}

\begin{proposition}
  Polynomials in the sense of definition \ref{def:depwtype} correspond
  precisely to polynomials with reductions where $R$ is the initial
  object in $\catc$.
\end{proposition}

\begin{proof}
  We draw attention to the fact that the coherence condition is
  vacuous when $R$ is initial. Aside from this it is obvious.
\end{proof}

\begin{definition}
  Suppose we are given a polynomial with reductions as in definition
  \ref{def:polyred}.
  
  We construct a pointed endofunctor $P_{f,g,h,k}$ as follows.

  Note that the coherence conditions gives us the isomorphism
  (equality, in fact) $\Sigma_g \Sigma_f \Sigma_k \cong \Sigma_h
  \Sigma_k$. We construct
  a map $\Sigma_h \Sigma_k k^\ast f^\ast \Pi_f h^\ast \rightarrow \operatorname{Id}_{\catc/Z}$
  as follows. Note that we have an evaluation map
  $f^\ast \Pi_f \rightarrow \operatorname{Id}_{\catc/Z}$ in $\catc/X$ (which is just the counit
  of the adjunction $f^\ast \dashv \Pi_f$). We also have a map
  $\Sigma_k k^\ast \rightarrow \operatorname{Id}_{\catc/Z}$ over $X$ given by the counit of the
  adjunction $\Sigma_k \dashv k^\ast$ (which recall is just one of the
  projection maps in the pullback). We have a similar such map for
  $h$. We put these together in the following composition:
  \begin{equation*}
    \Sigma_h \Sigma_k k^\ast f^\ast \Pi_f h^\ast \longrightarrow
    \Sigma_h \Sigma_k k^\ast h^\ast \longrightarrow
    \Sigma_h h^\ast \longrightarrow
    \operatorname{Id}_{\catc/Z}
  \end{equation*}

  Again using the counits of $\Sigma$ and pullback adjunctions we get
  a composition
  \begin{equation*}
    \Sigma_g \Sigma_f \Sigma_k k^\ast f^\ast \Pi_f h^\ast
    \longrightarrow
    \Sigma_g \Sigma_f f^\ast \Pi_f h^\ast \longrightarrow
    \Sigma_g \Pi_f h^\ast
  \end{equation*}

  Finally, we combine these together to get two maps out of $\Sigma_h
  \Sigma_k k^\ast f^\ast \Pi_f h^\ast$ in $\catc/Z$ and then take the
  pushout.
  \begin{equation}
    \label{eq:30}
    \begin{gathered}
      \xymatrix{ \Sigma_h \Sigma_k k^\ast f^\ast \Pi_f h^\ast \ar[r] \ar[d]
        & \operatorname{Id}_{\catc/Z} \ar[d]
        \\
        \Sigma_g \Pi_f h^\ast \ar[r] & P_{f,g,h,k} \pushoutcorner}
    \end{gathered}
  \end{equation}
  This defines a pointed endofunctor on $\catc/Z$ with the point given
  by the right hand inclusion of the pushout.

  We will refer to pointed endofunctors defined in this way as
  \emph{pointed polynomial endofunctors}.
\end{definition}

We first note that we get in this way a generalisation of Gambino and
Hyland's notion of dependent polynomial endofunctor in the following
proposition.

\begin{proposition}
  If $R$ is an initial object, then $P_{f,g,h,k}$ is just
  $P_{f,g,h} + 1$, which is a pointed endofunctor with a category of
  algebras isomorphic to the algebras of the dependent polynomial
  endofunctor on the underlying polynomial.
\end{proposition}

\begin{definition}
  Let $f, g, h, k$ be a polynomial with reductions. We refer to the
  initial object of the category of $P_{f, g, h, k}$-algebras (if it
  exists) as the \emph{$W$-type with reductions} on $f, g, h, k$.
\end{definition}

\begin{proposition}
  If $R$ is initial, then the $W$-type with reductions is just the
  dependent $W$-type on the underlying polynomial.
\end{proposition}

\subsection{A Formulation in the Internal Language of a Category}
\label{sec:form-intern-lang}

We will often work in the internal logic of $\catc$. In this case it
is useful to reformulate the definition in a more intuitive way as
follows. We will view $g \colon Y \rightarrow Z$ as a family of types
$Y_z$ indexed by $z \in Z$, and $f \colon X \rightarrow Y$ as a family
of types $X_{z, y}$ indexed by $z \in Z$ and $y \in Y_z$. We view $k$
as a family of types $R_{z, y, x}$ for $x \in X_{z, y}$.

We refer to $Y_z$ as the \emph{constructors} over $z \in Z$. For
$y \in Y_z$, we refer to $X_{z,y}$ as the \emph{arity} of the
constructor $y$. We will refer to the map $h \colon X \rightarrow Z$
as the \emph{reindexing map}.

Suppose we are given a family $(W_z)_{z \in Z}$ over $Z$. Now we can
reformulate the pointed polynomial endofunctor with reductions at $W$
as the following pushout using type theoretic notation as below.
\begin{equation*}
  \begin{gathered}
    \xymatrix{ \Sigma_{z : Z} \Sigma_{y : Y(z)} \Sigma_{r :
        R(y)} \Pi_{x : X(y)} W(h(x)) \ar[r] \ar[d] & W \ar[d]
      \\
      \Sigma_{z : Z} \Sigma_{y : Y(z)} \Pi_{x : X(y)} W(h(x))
      \ar[r] & P_{f,g,h,k}(W) \pushoutcorner}
  \end{gathered}
\end{equation*}

Then note that by the universal property of the pushout,
$P_{f, g, h, k}$-algebra structures on $W$ correspond precisely to
commutative triangles of the form below.
\begin{equation*}
  \begin{gathered}
    \xymatrix{ \Sigma_{z : Z} \Sigma_{y : Y(z)} \Sigma_{r :
        R(y)} \Pi_{x : X(y)} W(h(x)) \ar[drr]^{\lambda z, y, r,
        \alpha.\alpha(x)} \ar[d] & &
      \\
      \Sigma_{z : Z} \Sigma_{y : Y(z)} \Pi_{x : X(y)} W(h(x))
      \ar[rr]_c & & W}
  \end{gathered}
\end{equation*}
We can rephrase this as the following.
\begin{enumerate}
\item For each $z \in Z$, each constructor $y \in Y(z)$, and each
  element $\alpha$ of type $\Pi_{x : X(y)} W(h(x))$, we are given a
  choice of element $c(y, \alpha)$ of type $W(z)$.
\item For each $y \in Y_z$ and each $x \in X(y)$, if there exists
  $r \in R(x)$ then the equation $c(y, \alpha) = \alpha(x)$ is
  true. We refer to such equations as \emph{reduction equations} or
  just \emph{reductions}.
\end{enumerate}

\begin{remark}
  Note that the coherence condition ensures that whenever $y \in Y(z)$,
  $x \in X(y)$ and there exists $r \in R(x)$, we have $h(x) = g(f(x))$
  and so $\alpha(x)$ lies in the fibre $W(z)$, the same as
  $c(y, \alpha)$.
\end{remark}

The first part is then the same as an algebra structure over the
underlying polynomial endofunctor,
and the second part is what we gain by adding
reductions.

The $W$-type with reductions is then the object inductively generated
by the first condition subject to the equations in the second condition.
The way we combine an inductively defined type with equations in this
way is an example of a \emph{higher inductive type}. These play an
important role in homotopy type theory (see \cite{hottbook}).

In the above we only talked about $R(x)$ being inhabited, and didn't
need to depend on any particular choice of element from $R(x)$. We
justify this with the following proposition.
\begin{proposition}
  Every pointed
  polynomial endofunctor with reductions is isomorphic to one derived
  from a polynomial with reductions where
  where $k$ is monic. Moreover, given any polynomial with reductions,
  we obtain an isomorphic pointed endofunctor by replacing $k$ with
  the inclusion with its image in $X$.
\end{proposition}

\begin{proof}
  Recall that the image factorisation of $k$ is defined as the (unique
  up to isomorphism) factorisation of $k$ as a regular epimorphism followed
  by a monomorphism, as in the diagram below.
  \begin{equation*}
    \begin{gathered}
      \xymatrix{ R \ar[rr]^k \ar@{->>}[dr] & & X \\
        & S \ar@{>->}[ur]_{l} &}
    \end{gathered}
  \end{equation*}
  Note that this factorisation always exists since $\catc$ is locally
  cartesian closed and finitely cocomplete and therefore regular.

  The epimorphism $R \twoheadrightarrow R'$ then gives us an
  epimorphism $k^\ast f^\ast \Pi_f h^\ast \twoheadrightarrow l^\ast
  f^\ast \Pi_f h^\ast$, and so an epimorphism in the top left map
  below.
  \begin{equation*}
    \begin{gathered}
      \xymatrix{ \Sigma_h \Sigma_k k^\ast f^\ast \Pi_f h^\ast
        \ar@{->>}[r] \ar[dr]
        &
        \Sigma_h \Sigma_l l^\ast f^\ast \Pi_f h^\ast \ar[d] \ar[r] &
        \operatorname{Id}_{\catc/Z} \ar[d] \\
        & \Sigma_g \Pi_f h^\ast \ar[r] & P_{f, g, h, l} \pushoutcorner}
    \end{gathered}
  \end{equation*}
  However, now by diagram chasing the outer rectangle is also a
  pushout, and so $P_{f, g, h, k} \cong P_{f, g, h, l}$.
\end{proof}

\subsection{Coproducts of Pointed Polynomial Endofunctors with
  Reductions}
\label{sec:copr-point-polyn}

In \cite[Section 5]{gambinohylanddepw}, Gambino and Hyland observe
that under suitable conditions, the class of dependent polynomial
endofunctors over a fixed object $Z$ is closed under coproduct. We
will now show the analogous result when reductions are added. Note
that since we are now working with pointed endofunctors, the
appropriate notion of coproduct is the coproduct in the category of
pointed endofunctors, which appears in the category of endofunctors as
pushout along the units of the pointed endofunctors.

\begin{proposition}
  \label{prop:ppercoprod}
  Suppose that $\catc$ is a finitely cocomplete locally cartesian
  closed category with disjoint coproducts.\footnote{It's useful to
    note that every such category is extensive, as a corollary of
    \cite[Proposition 2.14]{carbonilackwalters}.}  Then the class of
  pointed polynomial endofunctors over a fixed object $Z$ is closed
  under coproduct.
\end{proposition}

\begin{proof}
  Suppose we are given two diagrams as below.
  \begin{equation*}
    \begin{gathered}
      \xymatrix { R_1 \ar[r]^{k_1} & X_1 \ar[r]^{f_1} \ar[dl]_{h_1} &
        Y_1 \ar[dr]^{g_1} \\
        Z & & & Z}    
    \end{gathered}
    \begin{gathered}
      \xymatrix { R_2 \ar[r]^{k_2} & X_2 \ar[r]^{f_2} \ar[dl]_{h_2} &
        Y_2 \ar[dr]^{g_2} \\
        Z & & & Z}    
    \end{gathered}
  \end{equation*}
  Similarly to the case for dependent polynomial endofunctors, we
  combine the two diagrams using coproduct as below.
  \begin{equation}
    \label{eq:sumppe}
    \xymatrix { R_1 + R_2 \ar[r]^{k_1 + k_2} & X_1 + X_2 \ar[r]^{f_1 +
      f_2} \ar[dl]^{[h_1, h_2]} &
      Y_1 + Y_2 \ar[dr]^{[g_1, g_2]} \\
      Z & & & Z}
  \end{equation}
  Again, by the same argument as for dependent polynomial
  endofunctors, note that
  $\Sigma_{[g_1,g_2]} \Pi_{f_1 + f_2} [h_1, h_2]^\ast \cong
  \Sigma_{g_1} \Pi_{f_1} h_1^\ast + \Sigma_{g_2} \Pi_{f_2} h_2^\ast$
  and
  $\Sigma_{[h_1,h_2]} \Sigma_{k_1 + k_2} (k_1 + k_2)^\ast (f_1 +
  f_2)^\ast \Pi_{f_1 + f_2} [h_1,h_2]^\ast \cong \Sigma_{h_1}
  \Sigma_{k_1} k_1^\ast f_1^\ast \Pi_{f_1} h_1^\ast + \Sigma_{h_2}
  \Sigma_{k_2} k_2^\ast f_2^\ast \Pi_{f_2} h_2^\ast$. Writing $P_i$
  for $\Sigma_{g_i} \Pi_{f_i} h_i^\ast$ and $Q_i$ for
  $\Sigma_{h_i} \Sigma_{k_i} k_i^\ast f_i^\ast \Pi_{f_i} h_i^\ast$ for
  $i = 1, 2$, we deduce that the pointed polynomial endofunctor
  generated by \eqref{eq:sumppe} is $\operatorname{Id}_{\catc/Z}
  \rightarrow S$ in the
  following pushout.
  \begin{equation*}
    \xymatrix{ Q_1 + Q_2 \ar[r] \ar[d] & \operatorname{Id}_{\catc/Z} \ar[d] \\
      P_1 + P_2 \ar[r] & S \pushoutcorner}
  \end{equation*}
  However, a quick diagram chase verifies that $\operatorname{Id}_{\catc/Z} \rightarrow S$ is
  the map produced by the following three pushouts.
  \begin{equation*}
    \begin{gathered}
      \xymatrix{ Q_1 \ar[r] \ar[d] & \operatorname{Id}_{\catc/Z} \ar[d] \\
        P_1 \ar[r] & S_1 \pushoutcorner}
    \end{gathered}
    \quad
    \begin{gathered}
      \xymatrix{ Q_2 \ar[r] \ar[d] & \operatorname{Id}_{\catc/Z} \ar[d] \\
        P_2 \ar[r] & S_2 \pushoutcorner}
    \end{gathered}
    \quad
    \begin{gathered}
      \xymatrix{ \operatorname{Id}_{\catc/Z} \ar[r] \ar[d] & S_1 \ar[d] \\
        S_2 \ar[r] & S \pushoutcorner}
    \end{gathered}
  \end{equation*}
  We deduce that the dependent pointed polynomial endofunctor produced by
  \eqref{eq:sumppe} (given by $\operatorname{Id}_{\catc/Z} \rightarrow
  S$) is the coproduct of
  the two diagrams given, as required.
\end{proof}

\section{Constructing $W$-Types with Reductions in $\Pi W$-Pretoposes}
\label{sec:constr-init-algebr}

\subsection{Review of Small Cover Bases and $\wisc$}
\label{sec:review-small-cover}

The axiom $\wisc$ was independently noticed and studied by various
authors.

For example, it was considered by Van den Berg in
\cite{vdbergpredtop} under the name $\mathbf{AMC}$, as a weakening of
the axiom $\mathbf{AMC}$ considered by Moerdijk and Palmgren in
\cite{moerdijkpalmgrenast1}. We recall the definition below and make
some basic observations that will be used later.

\begin{definition}
  Let $\catc$ be a category. A map $f \colon B \rightarrow A$ is a
  \emph{cover} if the only subobject of $A$ that it factors through is
  $A$ itself.
\end{definition}

\begin{proposition}
  If $\catc$ is a regular category then a map $f$ is a cover if and
  only if it is a regular epimorphism.
\end{proposition}

\begin{definition}
  Suppose we are given a square of the form below.
  \begin{equation}
    \label{eq:2}
    \begin{gathered}
      \xymatrix{ D \ar[r]^q \ar[d]_g & B \ar[d]^f \\
        C \ar[r]_p & A}      
    \end{gathered}
  \end{equation}
  We say the square is \emph{covering} if both $p$ and the canonical
  map $D \rightarrow B \times_A C$ are covers.
\end{definition}

In the internal logic of the category we can think of a covering
square as follows. We think of the map $f \colon B \rightarrow A$ as a
family of types indexed by $A$, which we write $(B_a)_{a \in A}$. We
think of the map $p \colon C \rightarrow A$ as a family of types
indexed by $A$, $(C_a)_{a \in A}$, where the requirement that $p$ is a
cover says that each $C_a$ is inhabited. We then think of the map
$g \colon D \rightarrow C$ as a family of types
$(D_{a, c})_{a \in A, c \in C_a}$. Finally, the requirement that the
canonical map $D \rightarrow B \times_A C$ is a cover says that for
every $a \in A$ and $c \in C_a$ we have a surjection
$q_{a, c} \colon D_{a, c} \twoheadrightarrow B_a$. Hence such a square
is sometimes referred to as a \emph{set of covers}.

\begin{definition}
  We say that a square as in \eqref{eq:2} is \emph{collection} if the
  following holds in the internal logic\footnote{Since the statement
    involves quantifying over a class of objects we need to use stack
    semantics to phrase it in the internal language. See
    e.g. the description by Roberts in \cite[Section 2]{robertswisc}
    for details.}. For all
  $a \in A$ and for each cover $e \colon E \twoheadrightarrow B_a$
  there is $c \in C_a$ and a map $t \colon D_c \rightarrow E$ such
  that $q_{a, c} = e \circ t$.
\end{definition}

Squares that are both covering and collection are sometimes referred
to as \emph{weakly initial sets of covers} or \emph{cover bases}.
\begin{definition}
  Let $\catc$ be a regular category. We say that a map
  $f \colon B \rightarrow A$ \emph{admits a cover base} if $f$
  fits into the right hand side of a square as in \eqref{eq:2} that is
  both covering and collection.

  The axiom \emph{weakly initial set of covers} ($\wisc$) states that
  any map admits a cover base.
\end{definition}

\begin{lemma}
  \label{lem:weakdepchoice}
  Suppose that we are given a covering collection square as in
  \eqref{eq:2}. Then the following holds in the internal language.

  For all $a \in A$, we have the following. Suppose we are given a
  family of types $(X_b)_{b \in B_a}$ such that $X_b$ is inhabited
  for all $b \in B_a$. Then there exists $c \in C_a$ and an element of
  the product type $\Pi_{d \in D_{a, c}} X_{q_{a, c}(d)}$.
\end{lemma}

\begin{proof}
  We apply collection to the cover $\Sigma_{b \in B_a} X_b
  \twoheadrightarrow B_a$ given by projection (which is a cover since
  each $X_b$ is inhabited).
\end{proof}

The following lemmas, which will be used later are easy to check, so
we omit proofs here.
\begin{lemma}
  \label{lem:coverbasepb}
  Suppose that a map $f \colon B \rightarrow A$ admits a weak cover
  base. Then the same is true for the pullback of $f$ along any map
  $h \colon A' \rightarrow A$.

  Moreover, the pullback of the covering and collection square along
  $h$ is also covering and collection.
\end{lemma}

\begin{lemma}
  \label{lem:coverbasecoprod}
  Suppose that $\catc$ has disjoint coproducts. Suppose that $f_1
  \colon B_1 \rightarrow A_1$ and $f_2 \colon B_2 \rightarrow A_2$
  both admit weak cover bases. Then the same is true for $f_1 + f_2
  \colon B_1 + B_2 \rightarrow A_1 + A_2$.

  Moreover, the  coproduct of the two covering and collection squares
  is itself covering and collection.
\end{lemma}

\subsection{Construction of the Initial Algebras}
\label{sec:constr-init-algebr-1}

In this section we work towards the construction of initial algebras
for dependent pointed polynomial endofunctors with reductions over
$\Pi W$-pretoposes. Although there are a number of possible approaches
to doing this that already appear in the literature, none seems to be
quite adequate for our purposes (this will be discussed further in
section \ref{sec:other-appr-constr}). The main obstacle is that we
wish for the construction to hold in categories that do not have
infinite colimits, such as realizability toposes.  We therefore give a
direct construction for $\Pi W$-pretoposes rather than applying an
existing result.

\subsubsection{Outline of the Construction}
\label{sec:outline-construction}

We start with a rough illustration of the overall idea, with the
motivation for each part of the proof.

For the proof to apply for
realizability toposes, the proof should be carried out in the internal
logic of the $\Pi W$-pretopos. We can see that some kind of
transfinite construction is likely to be necessary, and the only such
construction available to us internally is to use $W$-types (and in
section \ref{sec:non-existence-due} we will see that $W$-types really
are necessary for the theorem to hold). By the results of Gambino and
Hyland in \cite{gambinohylanddepw} we may use dependent $W$-types.
Some form of the axiom of choice may be necessary. $\wisc$ is acceptable,
since it holds in many examples of $\Pi W$-pretoposes including
realizability toposes, but we will try to avoid anything stronger.

The most na{\"i}ve approach using $W$-types is as follows. We know
from the description of $P_{f, g, h, k}$ algebras before that an
algebra structure on $W$ consists of the structure of an algebra over
the polynomial endofunctor $P_{f, g, h}$ whose operators satisfy the
reduction equations. We might therefore take $W$ to be an initial
algebra for $P_{f, g, h}$ and then simply quotient out by the
equivalence relation generated by the reduction equations. Note
however, that this won't work.  We need in particular an algebra
structure on $W / {\sim}$. For the time being we will consider the non
dependent case for simplicity.  Suppose that we want to define
$\sup(\alpha)$ for $\alpha \colon X_y \rightarrow W / {\sim}$ (the solid
horizontal line below). We want to use the algebra structure on $W$ to
define $\sup(\alpha)$, but to do this, we need a map
$X_y \rightarrow W$ (the dotted line below).
\begin{equation*}
  \xymatrix{ & W \ar@{->>}[d] \\
    X_y \ar@{.>}[ur] \ar[r]_\alpha & W/{\sim}}
\end{equation*}
In order for any such map to exist, we need the axiom of choice, and
then once we've found such a map we need to ensure that the particular
choice of map doesn't matter in order to produce a well defined
algebra structure.

Note however, that if $(A_i, q_i)_{i \in I}$ is a cover base for
$(X_y)_{y \in Y_z, z \in Z}$, then there does exist a dotted line in
the diagram below for some $i \in I$.
\begin{equation}
  \label{eq:lift}
  \begin{gathered}
    \xymatrix{ A_i \ar@{->>}[d] \ar@{.>}[r] & W \ar@{->>}[d] \\
      X_y \ar[r]_\alpha & W/{\sim}}
  \end{gathered}
\end{equation}

We therefore modify the na{\"i}ve argument as follows. We first form a
dependent $W$ type, using as arities, not $(X_y)_{y \in Y_z}$ directly,
but instead $(A_i)_{i \in I}$ where $(A_i, q_i)_{i \in I}$ is a cover
base for $(X_y)_{y \in Y_z}$.

We then define an equivalence relation $\sim$ on $W$ as (the image of)
another dependent $W$-type. We need to ensure of all of the following:
\begin{enumerate}
\item The reduction equations are satisfied.
\item If $\alpha(q_i(a)) \sim \alpha'(q_{i'}(a'))$ whenever $q_i(a) =
  q_{i'}(a')$ then also $\sup(\alpha) \sim \sup(\alpha')$ (function
  extensionality).
\item $\sim$ is an equivalence relation, in particular symmetric and
  transitive.
\end{enumerate}

Using a cover base like this has solved one problem but introduced
another.
In order to show that the algebra structure is initial, we
will need that any $\alpha \colon A_i \rightarrow W/{\sim}$ extends to
$X_y$ as below, but this is not always the case.
\begin{equation*}
  \xymatrix{ A_i \ar@{->>}[d] \ar[r]^\alpha & W \ar@{->>}[d] \\
    X_y \ar@{.>}[r] & W/{\sim}}
\end{equation*}
In fact the dotted line exists if and only if
$\alpha(q_i(a)) \sim \alpha(q_{i'}(a'))$ whenever
$q_i(a) = q_{i'}(a')$. To deal with this point we define $\sim$ not to
be an equivalence relation, but instead a \emph{partial} equivalence
relation. We then ensure that whenever
$\sup(\alpha) \sim \sup(\alpha)$ the condition above is satisfied (we
will refer to such elements as \emph{well defined}). Then we can
restrict to $w \in W$ such that $w \sim w$ in our construction.

A final point is that we know the dotted map in \eqref{eq:lift}
exists, but now we also have to show it is well defined.  We will
define $\sim$ as the image of a certain $W$-type, and well definedness
will amount to the existence of a function which provides for each
$a \in A_i$ and $a' \in A_{i'}$ such that $q_i(a) = q_{i'}(a')$, a
witness of $\alpha(a) \sim \alpha(a')$. We have effective quotients
and ensured that $\sim$ is an equivalence relation, but this only
tells us that such a witness exists for each $a$, not how to find one. To
deal with this, we use another cover base, this time for
$A_i \times_{X_y} A_{i'}$ over all $y \in Y_z$. We then can use the
same trick again of using the cover base in our dependent $W$-type
instead of $A_i \times_{X_y} A_{i'}$ itself.

We now provide a more careful, detailed version of the above argument.

\subsubsection{$2$-Cover Bases}
\label{sec:2-coverbases}

At the end of the outline we indicated that we would need two levels
of cover base. We formalise this using the following notion.

\begin{definition}
  Let $u \colon U \rightarrow I$ be a morphism in $\catc$.
  A \emph{$2$-cover base for $u$} consists of two squares of the
  following form that are both covering and collection.
\begin{equation}
  \label{eq:21}
  \begin{gathered}
    \xymatrix{ A \ar[r]^q \ar[d]_g & X \ar[d]^f \\
      J \ar[r]_p & Y }
  \end{gathered}
\end{equation}

\begin{equation}
  \label{eq:22}
  \begin{gathered}
    \xymatrix{ B \ar[r]^t \ar[d]_h & A \times_X A \ar[d]^{\langle g, g
        \rangle} \\
      K \ar[r]_s & J \times_Y J}
  \end{gathered}
\end{equation}
\end{definition}

Note in particular that if $\wisc$ holds in the pretopos, then any map
has a 2-cover base by applying $\wisc$ twice. Also if $X$ is the surjective
image of a projective object then $g \circ f$ has a 2-cover base, which
in particular includes all finite colimits of representables in
presheaf categories.

We also prove below that maps that admit $2$-cover bases are closed
under pullback and  coproduct.

\begin{lemma}
  \label{lem:twocovpb}
  Suppose that a map $f \colon X \rightarrow Y$ admits a $2$-cover
  base. Then the same is true for the pullback of $f$ along any map
  $Y' \rightarrow Y$.
\end{lemma}

\begin{proof}
  By applying lemma \ref{lem:coverbasepb} twice.
\end{proof}

\begin{lemma}
  \label{lem:twocovcoprod}
  Suppose that $\catc$ has disjoint coproducts. Suppose further that
  $f_1 \colon X_1 \rightarrow Y_1$ and
  $f_2 \colon X_2 \rightarrow Y_2$ admit $2$-cover bases. Then the
  same is true for $f_1 + f_2 \colon X_1 + X_2 \rightarrow Y_1 + Y_2$.
\end{lemma}

\begin{proof}
  By applying lemma \ref{lem:coverbasecoprod} twice.
\end{proof}

\subsubsection{The Underlying Object of the Initial Algebra}
\label{sec:underly-object-init}

We assume we are given a polynomial with reductions as in
\eqref{eq:27}, which as in section \ref{sec:form-intern-lang},
we view as families of types $Y_z$,
$X_{z, y}$ and $R_{z, y, x}$ (which we'll sometimes abbreviate to
$X_y$ and $R_x$).

We will assume that $f$ has a 2-cover base and view it as families of
types as follows. We assume we have a type $I_{z, y}$ for each
$z \in Z$ and $y \in Y_z$ together with a type $A_{z,y,i}$ (which we
will usually write just as $A_i$) and surjections
$q_i \colon A_i \twoheadrightarrow X_y$ such that $(A_i, q_i)_{i \in I}$ form a
cover base for $X_y$.

For the second part of the 2-cover base, we say that for each $y$ and
$z$ we have a type $J_{i, i'}$ for
each $i, i' \in I_{z, y}$ and a family of types and surjections $t_j
\colon B_j \twoheadrightarrow A_i \times A_{i'}$ for $j \in J_{i,
  i'}$, forming a cover base for $A_i \times_{X_y} A_{i'}$.

We will now construct the initial algebra.

We first define a family of types $W_z$ for $z \in Z$ as the dependent
$W$-type generated by the following rule:

If $y \in Y_z$, $i \in I_{z, y}$ and
$\alpha \in \Pi_{a \in A_i} W_{h(q_i(a))}$ then $\cons(y, i, \alpha)$
is a new element of $W_z$.

We now form a second dependent $W$-type, $Q$, which will be indexed
over $W \times_Z W$. First note that by the definition of
$W$ and the basic properties of dependent $W$-types, for every $w \in
W_z$ there is unique $y \in Y_z$, $i \in I_{z, y}$ and $\alpha \in
\Pi_{a \in A_i} W_{h(q_i(a))}$ such that $w = \cons(y, i, \alpha)$. We
will sometimes write $Q_{w_0, w_1}$ as $Q(w_0, w_1)$ to ease
readability.
\begin{enumerate}
\item If $w' \in W_z$, $q_1 \in Q_{w_0, w'}$ and
  $q_2 \in Q_{w', w_1}$, then $Q_{w_0, w_1}$ has an element of the
  form $\concat(q_1, q_2)$.
\item If $y$, $i$ and $\alpha$ are such that
  $w_0 = \cons(y, i, \alpha)$ and we are given $x \in X_{z, y}$ such
  that $\alpha(x) = w_1$, $r \in R_{z, y, x}$, $j \in J_{i, i}$, and
  $\gamma \colon \Pi_{b \in B_j} Q(\alpha(\pi_0(t_j(b))),
  \alpha(\pi_1(t_j(b))))$, then $Q_{w_0, w_1}$ has an element of the
  form $\mrgl(r, j, \alpha, \gamma)$.
\item If $y$, $i$ and $\alpha$ are such that
  $w_1 = \cons(y, i, \alpha)$ and we are given $x \in X_{z, y}$ such
  that $\alpha(x) = w_0$, $r \in R_{z, y, x}$, $j \in J_{i, i}$, and
  $\gamma \colon \Pi_{b \in B_j} Q(\alpha(\pi_0(t_j(b))),
  \alpha(\pi_1(t_j(b))))$, then $Q_{w_0, w_1}$ has an element of the
  form $\mrgr(r, j, \alpha, \gamma)$.
\item If we are given $y \in Y_z$, 
  $i_0, i_1 \in I_{y}$, $j \in J_{i_0, i_1}$,
  $\alpha_0 \in \Pi_{a \in A_{i_0}} W_{h(q_{i_0}(a))}$,
  $\alpha_1 \in \Pi_{a \in A_{i_1}} W_{h(q_{i_1}(a))}$ are such that
  $w_0 = \cons(y, i_0, \alpha_0)$ and $w_1 = \cons(y, i_1, \alpha_1)$
  and
  $\gamma \in \Pi_{b \in B_j} Q(\alpha_0(\pi_0(t_j(b))),
    \alpha_1(\pi_1(t_j(b))))$, then $Q_{w_0, w_1}$ has an element of the form
  $\ext(\alpha_0, \alpha_1, \gamma)$.
\end{enumerate}

We now define $Q_z := \Sigma_{w_0 \in W_z} \Sigma_{w_1 \in W_z}
Q_{w_0, w_1}$ and define $l, r \colon Q_z \rightarrow W_z$ to be the
two projections.

Note that we have defined $Q_z$ so that its image in $W_z \times W_z$,
which we write as $\sim$, is a partial equivalence relation. For
transitivity we use $\concat$. We prove symmetry in the following
lemma.

\begin{lemma}
  The relation $\sim$ on $W_z$ is symmetric.
\end{lemma}

\begin{proof}
  We show by induction on the construction of $Q$ that given any
  element $q$ of $Q_{w_0, w_1}$ we can prove there exists an element
  of $Q_{w_1, w_0}$. Formally, we need to be a little careful to make
  this argument work in general $\Pi W$-pretoposes. Write
  $\tau \colon W \times_Z W \rightarrow W \times_Z W$ for the map
  swapping the two components. Then we need to define a map from $Q$
  to $\tau^\ast(\sim)$, regarded as objects in $\catc/(W \times_Z
  W)$. We do this by defining an algebra structure on
  $\tau^\ast(\sim)$ and then using the initial map. The proof below is
  presented as an argument by induction on the structure of
  $Q_{w_0, w_1}$ because it's more intuitive, but it's easy to adapt
  to the form above.

  Note that the definitions of $\mrgl$ and $\mrgr$ were
  chosen so that they can just be swapped round, and $\concat$ is easy
  to deal with by induction.
  
  This only leaves us with the case of $\ext$, which is a little non
  trivial. Suppose we are given an element of $Q_{w_0, w_1}$ of the
  form $\ext(\alpha_0, \alpha_1, \gamma)$. Suppose further that we
  are given some $(a', a) \in A_{i'} \times_{X_y} A_{i}$. Then note
  that we also have $(a, a') \in A_i \times_{X_y} A_{i'}$.

  Since $t_j \colon B_j \twoheadrightarrow A_i \times_{X_z} A_{i'}$ is
  a surjection, there exists some $b \in B_j$ such that
  $t_j(b) = (a, a')$ and we have that
  $\gamma(b) \in \Pi_{b \in B_j} Q(\alpha_0(\pi_0(t_j(b))),
    \alpha_1(\pi_1(t_j(b))))$. By induction, we may assume therefore
  that $Q(\alpha_1(\pi_1(t_j(b))), \alpha_0(\pi_0(t_j(b))))$ contains
  some element $q'$. Then
  using the fact that $(B_{i',i,j})_{j \in J_{i', i}}$ is a cover
  base, we deduce that there exists $j \in J_{i' , i}$ together with
  $\gamma' \colon B_{i',i',j} \rightarrow Q'$ choosing witnesses of
  this. We then form the element of $Q_{w_1, w_0}$,
  $\ext(\alpha_1, \alpha_0, \gamma')$ and note that it is as required.
\end{proof}

We say that $w \in W$ is \emph{well defined} if $w \sim w$. We write
$W'$ for the set of well defined elements of $Z$. Note that $\sim$
restricts to an equivalence relation on $W'$ (as is always the case
for partial equivalence relations). Note that we can use $\ext$ to
produce well defined elements as follows.

\begin{lemma}
  \label{lem:welldefdlem}
  Suppose that $w, w' \in W$ are of the form $\cons(y, i, \alpha)$ and
  $\cons(y, i', \alpha')$ respectively. Suppose further that for every
  $(a, a') \in A_i \times_{X_y} A_{i'}$ we have that
  $\alpha(a) \sim \alpha'(a')$. Then $w \sim w'$.
\end{lemma}

\begin{proof}
  Suppose that for every $(a, a') \in A_i \times_{X_y} A_{i'}$ we have
  that $\alpha(a) \sim \alpha'(a')$. Then using the fact that
  $(B_{i, i', j})_{j \in J_{i, i'}}$ is a cover base for
  $A_i \times_{X_y} A_{i'}$, there exists $j \in J_{i, i'}$ and a
  choice function
  $\gamma \in \Pi_{b \in B_j} Q(\alpha(\pi_0(t_j(b))),
  \alpha'(\pi_1(t_j(b))))$. We then have
  $\ext(\alpha, \alpha', \gamma) \in Q(\cons(y, i, \alpha), \cons(y,
  i', \alpha')$ and so $w \sim w'$.
\end{proof}

\begin{lemma}
  Suppose $w \in W$ is of the form $\cons(y, i, \alpha)$ and for every
  $(a, a') \in A_i \times_{X_y} A_i$ we have that
  $\alpha(a) \sim \alpha(a')$. Then $w \sim w$.
\end{lemma}

\begin{proof}
  This is a special case of the previous lemma where $\alpha =
  \alpha'$ and $i = i'$.
\end{proof}

\subsubsection{The Algebra Structure of the Initial Algebra}
\label{sec:algebra-struct-init}

We now give $W'/{\sim}$ an algebra structure over the pointed
endofunctor. We first show the following lemma.

\begin{lemma}
  Suppose that we are given a map $\alpha_0 \in \Pi_{x \in X_y} W_{h(x)}'
  / {\sim}$. Then there exists $i \in I$ and $\alpha \in \Pi_{a \in A_i}
  W'$ such that for all $a \in A_i$ we have $[\alpha(a)] =
  \alpha_0(q_i(a))$.

  Furthermore, if $(i, \alpha)$ and $(i', \alpha')$ are two such pairs
  then $\cons(i, \alpha) \sim \cons(i', \alpha')$ (and in particular
  these are well defined).
\end{lemma}

\begin{proof}
  First we construct $\alpha$ by applying lemma
  \ref{lem:weakdepchoice}.

  Now suppose that $(i, \alpha)$ and $(i', \alpha')$ are two such
  pairs. By lemma \ref{lem:welldefdlem} it suffices to show that
  $\alpha(a) \sim \alpha'(a')$ for every $(a, a') \in A_i \times_{X_y}
  A_{i'}$. However, we know that
  $[\alpha(a)] = \alpha_0(q_i(a))$ and $[\alpha'(a')] =
  \alpha_0(q_{i'}(a'))$. Since $(a, a')$ belongs to the pullback over
  $X_y$, we have $q_i(a) = q_{i'}(a')$, and so $[\alpha(a)] =
  [\alpha'(a')]$. Finally, since quotients are effective, we deduce
  $\alpha(a) \sim \alpha'(a')$.
\end{proof}

\begin{lemma}
  \label{lem:rwalgstr}
  We exhibit an algebra structure on $W'/{\sim}$ over the pointed
  endofunctor.
\end{lemma}

\begin{proof}
  By the characterisation of algebra structures in section
  \ref{sec:form-intern-lang}, it suffices to construct
  $\sup(\alpha_0)$ for every
  $\alpha_0 \in \Pi_{x \in X_y} W_{h(x)}'/ {\sim}$ and show that it
  respects the reduction equations.

  Given $\alpha_0 \in \Pi_{x \in X_y} W_{h(x)}'/ {\sim}$, we define
  $\sup(\alpha_0)$ to be $[ \cons(i, \alpha) ]$ where $(i, \alpha)$ is
  such that $[\alpha(a)] = \alpha_0(q_i(a))$ for every $a \in
  A_i$. This determines a unique element of $W_z / {\sim}$ by lemma
  \ref{lem:welldefdlem}.

  We now need to show that, for all $x \in X$, if $R_{z, y, x}$ is
  inhabited, then $\sup(\alpha_0) = \alpha_0(x)$. To do this, we will
  show there exists an appropriate element of $Q$ using
  $\mrgl$. Firstly, let $i$ and $\alpha$ be as above. Let $a \in A_i$
  be such that $q_i(a) = x$. Next, note that following the proof of
  lemma \ref{lem:welldefdlem} we can show there exists
  $j \in J_{i, i}$ and $\gamma \colon B_j \rightarrow Q'$ such that
  for all $b \in B_j$,
  $\gamma(b) \in Q(\alpha(\pi_0(t_j(b))),
  \alpha(\pi_1(t_j(b))))$. Then, $\mrgl(a, j, \alpha, \gamma)$
  witnesses $\cons(i, \alpha) \sim \alpha(a)$ and so
  $\sup(\alpha_0) = [\cons(i, \alpha)] = [\alpha(a)] = \alpha_0(x)$ as
  required.
\end{proof}

\subsubsection{Proof of Initiality}
\label{sec:proof-initiallity}

We now show that the algebra structure we defined is initial. Suppose
that we are given an object $T$ together with an algebra structure on
$T$. We will use the presentation from section
\ref{sec:form-intern-lang}, where we view an algebra structure as an
algebra structure for the underlying polynomial, $c \colon \Sigma_g
\Pi_f h^\ast(T) \rightarrow T$ such that $c$ respects the reduction
equations.

We first need to construct algebra map from $W'/{\sim}$ to $T$, and then
show that it is unique.

For this, we will follow the basic outline below.
\begin{enumerate}
\item Define a relation $S \rightarrowtail W \times_Z T$ by induction
  on the construction of $W$.
\item Show by induction on the construction of $Q$ that for every
  $q \in Q_{w_0, w_1}$ there exists a unique $t \in T$ such that
  $\langle w_0, t \rangle \in S$ and the same $t$ is unique such that
  $\langle w_1, t \rangle \in S$ (which in particular tells us that
  when $w \sim w$ there exists a unique $t \in T$ such that
  $\langle w_0, t \rangle \in S$).
\item Deduce (using effectiveness of quotients) that the corresponding
  relation on $W'/{\sim} \times_Z T$ is functional, and so gives a morphism
  $W'/{\sim} \rightarrow T$ over $Z$.
\end{enumerate}

We define $S \rightarrowtail W \times_Z T$ inductively as follows.

We add $\langle \cons(i, \alpha), x \rangle$ to $S$ when
$\alpha' \in \Pi_{x \in X_z} T_{h(x)}$ is such that for every
$a \in A_i$, $\alpha'(q(a))$ is the unique $t$ such that
$\langle \alpha(a), t \rangle \in S$ and $x$ is the result of applying
the algebra structure of $T$ to $\alpha'$.

Formally, we can construct $S$ in an arbitrary $\Pi W$-pretopos as a
dependent $W$-type as follows. We work over the context $W \times_Z
T$.

Let $\langle w, t \rangle \in W \times_Z T$. We construct $S(w, t)$
as follows.
Suppose we are given all of the following.
\begin{enumerate}
\item A triple $y, i, \alpha$ such that $w = \sup(y, i, \alpha)$
\item A dependent function $\alpha' \colon \Pi_{x : X(z)} T(h(x))$
  such that $t = c(\alpha')$ (recall that $c$ is the algebra structure
  for $T$).
\item A dependent function $\beta \colon \Pi_{a : A_i}S(\alpha(a),
  \alpha'(q(a)))$
\end{enumerate}
Then we construct a new element of $S$ of the form $\sup(\alpha',
\beta)$.

One can check that the
composition $S \rightarrow W \times_Z T \rightarrow W$ is monic, it
follows that this definition of $S$ matches the other definition.

We can now state and prove the main lemma.

\begin{lemma}
  Let $T$ and $S$ be as above. Then for any $e \in Q_{w_0, w_1}$,
  there exists a unique $t$ such that $\langle w_0, t \rangle \in S$
  and the same $t$ is unique such that $\langle w_1, t \rangle \in S$.
\end{lemma}

\begin{proof}
  We prove this by induction on the construction of
  $e \in Q_{w_0, w_1}$.

  The case $\concat$ is easy to deal with by induction.

  We next consider $\ext$. Suppose that
  $w_0 = \cons(y, i_0, \alpha_0)$, $w_1 = \cons(y, i_1, \alpha_1)$ and
  $e$ is of the form $\ext(\alpha_0, \alpha_1, \gamma)$. Note that we
  may assume by induction that for every $b \in B_j$, $\gamma(b)$
  satisfies the statement of the lemma. We define an element
  $\tilde{\alpha}$ of $\Pi_{x \in X_z} T_{h(x)}$ as follows. Given
  $x \in X_z$, let $a$ be such that $q_{i_0}(a) = x$ (which exists
  since $q_{i_0}$ is surjective). Furthermore, let $a'$ be such that
  $q_{i_1}(a') = x$. Then clearly
  $(a, a') \in A_{i_0} \times_X A_{i_1}$. Let $b$ be such that
  $t_j(b) = (a, a')$. Then
  $\gamma(b) \in Q(\alpha_0(\pi_0(a)), \alpha_1(\pi_1(a')))$, so in
  particular there is a unique $t \in T_{h(x)}$ such that
  $\langle \alpha(a), t \rangle \in S$. We will take
  $\tilde{\alpha}(x)$ to be such a $t$, but we still need to complete
  the proof that $t$ is uniquely determined by $x$. It only remains to
  check that $t$ is independent of the choice of $a \in
  q_i^{-1}(x)$. So let $a'' \in q_i^{-1}(x)$ and let $t''$ be unique
  such that $\langle \alpha(a''), t'' \rangle \in S$. We need to check
  that $t = t''$. Suppose that $a' \in q_{i'}^{-1}(x)$, as before, and
  note that we have $b, b' \in B_j$ such that
  $\gamma(b) \in Q(\alpha(a), \alpha'(a'))$ and
  $\gamma(b') \in Q(\alpha(a''), \alpha'(a'))$. Using the inductive
  hypothesis, we have then a unique $t'$ such that
  $\langle \alpha'(a'), t' \rangle \in S$ and $t = t'$ and $t'' = t'$,
  which implies $t = t''$, as required. Finally, note that the
  $\tilde{\alpha}$ we have now defined is unique such that for all
  $a \in A_i$,
  $\langle \alpha(a), \tilde{\alpha}(q_i(a)) \rangle \in S$. By the
  same argument as above, $\tilde{\alpha}$ is also unique such that
  for all $a \in A_{i'}$,
  $\langle \alpha'(a), \tilde{\alpha}(q_{i'}(a)) \rangle \in
  S$. Therefore, applying the algebra structure of $T$ to
  $\tilde{\alpha}$ gives us a unique $t$ such that
  $\langle w_0, t \rangle \in S$ and the same $t$ is unique such that
  $\langle w_1, t \rangle \in S$ as required.

  The last two cases to consider are $\mrgl$ and $\mrgr$. We will just
  consider when $e$ is of the form $\mrgl(a_0, j, \alpha, \gamma)$,
  the other case being similar.

  First note that by induction we may assume that for every $b \in
  B_j$, $\gamma(b)$ satisfies the statement of the
  lemma. Hence, we may apply the same argument as before to construct
  a unique $\alpha_0 \in \Pi_{x \in X} T_{h(x)}$ such that for all $a
  \in A_i$, $\langle \alpha(a), \alpha_0(q_i(a)) \rangle \in
  S$.\footnote{In 
    fact, this is the sole reason for including $\gamma$ in the
    definition of $\mrgl$.}
  We now have, as before that applying the
  algebra structure of $T$ to $\alpha_0$ gives us a unique $t$ such
  that $\langle \cons(i, \alpha), t \rangle \in S$.

  Also, note that there exists $b \in B_j$ such that $q_j(b) = (a_0,
  a_0)$, and so again by induction, there is a unique $t'$ such that
  $\langle \alpha(a_0), t' \rangle \in S$.

  Finally, since the algebra structure on $T$ has to respect the
  reduction equations, we have $t = t'$, as required.
\end{proof}

Finally, since $W'$ includes only the well defined elements of $W$, we
deduce that for every $w \in W'$, there is a unique $t \in T$ such
that $\langle w, t \rangle \in S$, and if $w \sim w'$ and $t'$ is
unique such that $\langle w', t' \rangle \in S$ then $t = t'$. We
deduce that this gives us a well defined function
$W'/{\sim} \rightarrow T$. Finally note that by the definition of $S$
and the algebra structure on $W'/{\sim}$, we can easily see that the
function is the unique algebra structure preserving map, which gives
us the lemma below.

\begin{lemma}
  $W'/{\sim}$ with the algebra structure given in lemma
  \ref{lem:rwalgstr} is initial.
\end{lemma}

We can now deduce the main theorem of this section.
\begin{theorem}
  \label{thm:rwinitialalg}
  Let $\catc$ be a $\Pi W$-pretopos.
  \begin{enumerate}
  \item Suppose we are given a polynomial with reductions in $\catc$
    together with a $2$-covering for it. Then we can construct an
    initial algebra for the corresponding pointed polynomial
    endofunctor.
  \item Suppose that $\wisc$ holds in $\catc$, making it a predicative
    topos. Then every pointed polynomial
    endofunctor admits an initial algebra. In other words, $\catc$ has
    all $W$-types with reductions.
  \end{enumerate}
\end{theorem}

\section{A Simplification in Categories of Presheaves}
\label{sec:simpl-categ-presh}

In section \ref{sec:constr-init-algebr} we gave a very general
construction that works for any polynomial with reductions in any
predicative topos. However, the result is in some ways
unsatisfactory. Since we relied on effective quotients, the result
does not apply to presheaf assemblies, which are one of the main
intended applications of this work. The reliance on
cover bases and $\wisc$ may turn out to be less serious in practice,
but is still not ideal. It could, for example lead to subtle coherence
issues when applying the results to the semantics of type theory.

In this section we therefore give another version of the main result,
which will appear as theorem \ref{thm:initialalgps}. We no longer
assume effective quotients or $\wisc$, so the result is applicable to
a wider range of categories, and we obtain more concrete descriptions
of the initial algebras. The class of polynomials with reductions that
we consider is, however, much more restricted, but will still include
many interesting examples.

Recall, e.g. from \cite[Chapter 7]{jacobs} that in any finitely
complete category we can define the notion of internal category, and
thereby a notion of category of internal diagrams (which we will refer
to here as internal presheaves).

Let $\catc$ be a finitely cocomplete locally
cartesian closed category with disjoint coproducts and $W$-types
(e.g. a category of
assemblies). Note that for any internal category $\smcat{C}$ in
$\catc$, the category of internal assemblies is also finitely
cocomplete locally cartesian closed, and has disjoint coproducts.
We will construct
initial algebras for a certain class of polynomials with reductions in
such internal presheaf categories.

\subsection{Dependent $W$-Types in Internal Presheaves}
\label{sec:dependent-w-types}

We first give an explicit description of dependent $W$-types in
presheaves. We will consider polynomial endofunctors over the
following polynomial in internal presheaves. Note that by forgetting the
action, we can also view this as a polynomial in $\catc/\smcat{C}_0$
\begin{equation*}
  \xymatrix{ & X \ar[dl]_h \ar[r]^f & Y \ar[dr]^g & \\
    Z & & & Z}
\end{equation*}

Suppose we are given a morphism of presheaves $A \rightarrow Z$. Then,
using (the internal version of) Yoneda and the adjunctions
$f^\ast \dashv \Pi_f$ and $\Sigma_h \dashv h^\ast$ we can show that
for $c \in \smcat{C}_0$ elements of $\Pi_f h^\ast(A)(c)$ consist of
$y \in Y(c)$ (which we view as a map
$\ulcorner y \urcorner \colon \yoneda(c) \rightarrow Y$) together with
with a map $f^\ast(\yoneda(c)) \rightarrow A$ making the following
square commute.
\begin{equation*}
  \xymatrix{ f^\ast(\yoneda(c)) \ar[r] \ar[d]_{f^\ast(\ulcorner z
      \urcorner)} & A \ar[d] \\
    Y \ar[r]_h & Z }
\end{equation*}
Expanding the definition of $\yoneda(c)$, we see that this consists of
a (dependent) function assigning, for each $d \in \smcat{C}_0$, each
$\sigma \colon c \rightarrow d$ in $\smcat{C}_1$, and each
$x \in f_d^{-1}(Y(\sigma)(y))$, an element, $\alpha(\sigma, x)$ of
$A(d, h_d(Y(\sigma)(y)))$, satisfying the \emph{naturality condition}
that for all $\tau \colon d \rightarrow d'$ we have
$\alpha(\tau \circ \sigma, X(\tau)(x)) = A(\tau)(\alpha(\sigma,
x))$. Note that if we drop the naturality condition, then we get a
dependent polynomial functor in $\catc$. We denote the corresponding
dependent $W$-type as
$W_0$. We define the action of morphisms making $W_0$ into a presheaf
over $Z$ as follows. For $c \in \smcat{C}_0$ and $z \in Z(c)$,
everything in $W_0(c, z)$ is of the form
$\sup(y, \alpha)$ where $y$ and $\alpha$ are as above. Given
$\tau \colon c \rightarrow c'$, we define
$W_0(\tau) (\sup(y, \alpha))$ to be $\sup(Y(\tau)(y), \alpha')$ where
$\alpha'(\sigma, x)$ is defined to be $\alpha(\sigma \circ \tau, x)$.
Following Moerdijk and Palmgren in \cite[Paragraph
5.4]{moerdijkpalmgrenwtypes} we note that if we can form the subobject
of $W_0$ consisting of the corresponding dependent $W$-type consisting
of hereditarily natural\footnote{In \cite{moerdijkpalmgrenwtypes}
  Moerdijk and Palmgren refer to another condition in addition to
  naturality that they call \emph{composability}. We have already
  dealt with this by using exploiting the fact that we are using
  dependent $W$-types rather than ordinary $W$-types.} elements, then
this gives the $W$-type in presheaves. We can construct this subobject
in an arbitrary locally cartesian closed category with $W$-types by a
similar technique to the construction of dependent $W$-types from
ordinary $W$-types, which we do in the following lemma.

\begin{lemma}
  \label{lem:herednat}
  $W_0$ has a subobject $W$ such that an element $\sup(y, \alpha)$ of
  $W_0$ belongs to $W$ if and only if $\alpha$ is natural, and for
  every $\sigma \colon c \rightarrow d$ and $x \in Y(\sigma)(y)$, we
  have $\alpha(\sigma, x) \in W$.
\end{lemma}

\begin{proof}
  We first modify the definition of $W_0$ to get a dependent $W$-type,
  $V$ defined as follows. We take the context and the constructors to
  be the same as for $W_0$. For $W_0$, the arity at $Y \in Y(c, z)$
  consisted of pairs $(\sigma, x)$ where
  $\sigma \colon c \rightarrow d$ and $x \in X(d, Y(\sigma)(y))$. For
  $V$, we instead define an element of the arity over $y$ to consist
  of two morphisms $\sigma \colon c \rightarrow d$ and
  $\tau \colon d \rightarrow e$ in $\smcat{C}$, together with
  $x \in X(d, Y(\sigma)(y))$. We define the reindexing map at
  $(\sigma, \tau, x)$ to be $Z(\tau \circ \sigma)(z)$. In other words
  we add an element to $V(c, z)$ of the form $\sup(y, \alpha)$ whenever
  $y \in Y(c, z)$, and $\alpha$ is a dependent function such that for
  $\sigma \colon c \rightarrow d$, $\tau \colon d \rightarrow e$ and
  $x \in X(d, Y(\sigma)(y))$, $\alpha(\sigma, \tau, x)$ is an element
  of $V(e, Z(\tau \circ \sigma)(z))$.

  Note that we have two maps $r, s \colon W_0 \rightarrow V$ over $Z$
  defined recursively as follows. Suppose we are given an element of
  $W_0$ of the form $\sup(y, \alpha)$. We define $r(\sup(y, \alpha))$
  to be $\sup(y, \alpha')$ and $s(\sup(y, \alpha))$ to be
  $\sup(y, \alpha'')$, where $\alpha'$ and $\alpha''$ are defined as
  follows. Let $\sigma \colon c \rightarrow d$,
  $\tau \colon d \rightarrow e$ and $x \in X(d, Y(\sigma)(y))$. We
  define $\alpha'(\sigma, \tau, x)$ to be
  $r(\alpha(\tau \circ \sigma, X(\tau)(x)))$. We define
  $\alpha''(\sigma, \tau, x)$ to be $s(W_0(\tau)(\alpha(\sigma, x)))$.

  We define $W$ to be the equaliser of $r$ and $s$.

  Note that $r$ and $s$ have a common retract $t \colon V \rightarrow
  W_0$ defined recursively as follows. Given an element of $V$ of the form
  $\sup(y, \alpha)$, we define $t(\sup(y, \alpha))$ to be $\sup(y,
  \alpha')$ where $\alpha'$ is defined as follows. Given $\sigma
  \colon c \rightarrow d$ and $x \in X(d, Y(\sigma)(y))$, we define
  $\alpha'(\sigma, x) := t(\alpha(\sigma, 1_d, x))$.

  We now need to check that $W$ does in fact satisfy the lemma.

  Every element of $W_0$ is of the form
  $\sup(y, \alpha)$. First suppose that $\sup(y, \alpha) \in W$. Then
  $\alpha' = \alpha''$, where $\alpha'$ and $\alpha''$ are as
  above. Hence for all $\sigma \colon c \rightarrow d$,
  $\tau \colon d \rightarrow e$ and $x \in X(d, Y(\sigma)(y))$ we have
  $r(\alpha(\tau \circ \sigma, X(\tau)(x))) =
  s(W_0(\tau)(\alpha(\sigma, x)))$. Applying the common retract $t$ of
  $r$ and $s$ to this equation allows us to deduce
  $\alpha(\tau \circ \sigma, X(\tau)(x)) =
  W_0(\tau)(\alpha(\sigma, x))$ for all $\sigma, \tau, x$, and so
  that $\alpha$ is natural. Applying the equation to the special case
  $\tau = 1_d$, allows to deduce $r(\alpha(\sigma,
  x)) = s(\alpha(\sigma, x))$ and so $\alpha(\sigma, x) \in W$ for all
  $\sigma$ and $x$.

  Conversely, suppose that $\alpha$ is natural and
  $\alpha(\sigma, x) \in W$ for all $\sigma$ and $x$. We need to show
  that $\alpha' = \alpha''$ where $\alpha'$ and $\alpha''$ are as
  above. Naturality tells us that for all
  $\sigma \colon c \rightarrow d$, $\tau \colon d \rightarrow e$ and
  $x \in X(d, Y(\sigma)(y))$ we have
  $\alpha(\tau \circ \sigma, X(\tau)(x)) = W_0(\tau)(\alpha(\sigma,
  x))$, and so applying $s$ we have
  $s(\alpha(\tau \circ \sigma, X(\tau)(x))) =
  s(W_0(\tau)(\alpha(\sigma, x)))$. However, we also have
  $\alpha(\tau \circ \sigma, X(\tau)(x)) \in W$ and so
  $s(\alpha(\tau \circ \sigma, X(\tau)(x))) = r(\alpha(\tau \circ
  \sigma, X(\tau)(x)))$. Putting these together we have
  $r(\alpha(\tau \circ \sigma, X(\tau)(x))) =
  s(W_0(\tau)(\alpha(\sigma, x)))$ and so $\alpha' = \alpha''$, and so
  $\sup(y, \alpha) \in W$, as required.
\end{proof}

\begin{remark}
  Lemma \ref{lem:herednat} can also be proved using the notion of
  paths, as used by Van den Berg and De Marchi for $M$-types in
  \cite[Proposition 5.7]{vdbergdemarchi}.
\end{remark}

Now note that the action of
morphisms restricts to the subobject $W$, making $W$ into a presheaf
(and in fact a subpresheaf of $W_0$). We can then assign $W$ an
algebra structure making into the initial algebra for the polynomial
endofunctor.

\subsection{Decidable and Locally Decidable Polynomials with
  Reductions}
\label{sec:decid-locally-deci}

We now define the class of polynomials with reductions that we will
work over. The basic idea is that a polynomial is decidable when for
each constructor there is either no reduction at all, or there is
exactly one reduction. $W$-types with reductions over decidable
polynomials can be viewed directly as dependent $W$-types. This makes
them simple to construct but not so useful in practice when we already
have $W$-types.

Therefore, instead of decidable polynomials with reductions, we look
at \emph{locally decidable} polynomials with reductions. In this case
we work in an internal presheaf category, and then the polynomial does
not have to be decidable in the internal logic of the presheaf
category. It turns out to be sufficient that it is decidable in the
external category, in order to construct the initial algebras.

\begin{proposition}
  \label{prop:decpolyequivs}
  The following are equivalent.
  \begin{enumerate}
  \item The polynomial with reductions \eqref{eq:27} is isomorphic to
    one of the following form.
    \begin{equation}
      \label{eq:decpolydiagram}
      \begin{gathered}
        \xymatrix { Y_2 \ar[r]^{\iota_2} & X_1 + Y_2 \ar[r]^{f_1 + 1_{Y_2}}
          \ar[dl]_h & Y_1 + Y_2 \ar[dr]^g \\
          Z & & & Z}        
      \end{gathered}
    \end{equation}
  \item $f \circ k$ is isomorphic to one of the inclusion maps of a
    coproduct.
  \item $f \circ k$ is a monomorphism with decidable image.
  \item In the internal logic, the following holds. For each
    constructor $y \in Y_z$, either there are no $x \in X_y$ such that
    $R_{z, y, x}$ is inhabited, or there exists exactly one
    $x \in X_{z, y}$ such that $R_{z, y, x}$ is inhabited, and in this
    case $R_{z, y, x}$ also has exactly one element.
  \end{enumerate}
\end{proposition}

\begin{definition}
  We say a polynomial with reductions is
  \emph{decidable} if it satisfies one of
  the equivalent conditions in proposition \ref{prop:decpolyequivs}.
\end{definition}

\begin{definition}
  When we are working in the internal logic of the locally cartesian
  closed category, and $y \in Y_z$, we will say \emph{$y$ does not
    reduce} if $R_{z, y, x}$ is empty for all $x$, and we will say
  \emph{$y$ reduces at $x$} if $x$ is unique such that $R_{z, y, x}$
  is inhabited.
\end{definition}

\begin{definition}
  We say a polynomial with reductions in presheaves is
  \emph{locally decidable} if its image in $\catc/\smcat{C}_0$ after
  forgetting the action is decidable.
\end{definition}

\begin{proposition}
  \label{prop:locdecboolean}
  Suppose that $\catc$ is a boolean topos with natural number
  object. Then a pointed polynomial endofunctor is decidable if and
  only if $f \circ k$ is monic. Similarly if $\catc$ is a category
  internal presheaves over a boolean topos with natural number object,
  then a pointed polynomial endofunctor is locally decidable if and
  only if $f \circ k$ is monic.
\end{proposition}

Given a polynomial with reductions in a category of presheaves, it
makes sense to talk about it being locally decidable and it also makes
sense to talk about the polynomial with reductions being decidable
internally in the category of presheaves.
It's important to note the distinction between the
two notions.

Every decidable polynomial with reductions is also
locally decidable, but the converse does not hold in general.
Given a morphism $\sigma \colon c \rightarrow d$ in the
internal category $\smcat{C}$, locally decidability says that any
$y \in Y(c)$ either lies in the image of $f_c \circ k_c$ or does not,
and the same for $y \in Y(d)$. In any case we know that if
$y \in Y(c)$ belongs to the image of $f_c \circ k_c$ then also
$Y(\sigma)(y)$ belongs to the image of $f_d \circ k_d$. Decidability
states that the converse also holds, so if $Y(\sigma)(y)$ lies in the
image of $f_d \circ k_d$, then $y$ lies in the image of
$f_c \circ k_c$. In order to get a result applicable to the CCHM model
of type theory, we need it to apply to locally decidable pointed
polynomial endofunctors that aren't decidable. Explicitly, we need to
allow for the case of $y \in Y(c)$ that does not belong to the image
of $f_c \circ k_c$ but where $Y(\sigma)(y)$ does belong to the image
of $f_d \circ k_d$, or informally ``$\sup(y, \alpha)$ does not yet
reduce at $c$, but will reduce at $d$.''

\subsection{Construction of the Initial Algebras}
\label{sec:constr-init-algebr-2}

Assume we are given a polynomial with reductions of the form
\eqref{eq:27} that is locally decidable. We will construct an initial
algebra for the corresponding pointed endofunctor, showing that
$W$-types with reductions exist for all locally decidable polynomials
with reductions (theorem \ref{thm:initialalgps}).

\subsubsection{Normal Forms}
\label{sec:norm-forms-norm}

We first form a variant of the dependent $W$-type $W_0$ that we used
in the construction of dependent $W$-types in presheaves. We call this
$N_0$, and define it as follows. For $c \in \smcat{C}_0$ and
$z \in Z(c)$, we add an element $\sup(y, \alpha)$ to $N(c, z)$
whenever $y \in Y(c, z)$ with $y \notin \im(f \circ k)$ and
$\alpha \in \Pi_{d \in \smcat{C}_0} \Pi_{\sigma \colon c \rightarrow
  d} \Pi_{x \in X(c, z, y)} N_0(d, h_d(Y(\sigma)(y)))$. For the moment
we don't add any naturality condition. Note that if $W_0$ is the
corresponding $W$-type over all elements of $Y$ (again, with the
naturality condition dropped), then we have a canonical monomorphism
$i \colon N_0 \rightarrow W_0$ over $Z$. We refer to elements of $N_0$
as \emph{normal forms}. In other words we only
consider those terms that do not reduce because they have constructor
$y \in Y$ whose fibre over $f \circ k$ is empty. Like with $W_0$, we
can define for each $\tau \colon c \rightarrow c'$ and each
$z \in Z(c)$, a map
$N_0(\tau) \colon N_0(c, z) \rightarrow N_0(c', Z(\tau)(z))$. Any
element of $N_0(c, z)$ is of the form $\sup(y, \alpha)$. Define
$\alpha'$ the same as for $W_0$. Note that $\sup(Y(\tau)(y), \alpha')$
is not necessarily an element of $N_0(c', Z(\tau)(z))$, since
$Y(\tau)(y)$ might reduce. However, by local decidability we
can split into two cases: either $Y(\tau)(y)$ reduces
or it does not. If it does not, we define
$N_0(\tau)(\sup(y, \alpha))$ to be $\sup(Y(\tau)(y), \alpha')$, the
same as for $W_0$. If $Y(\tau)(y)$ reduces, at $x$, say,
define $N_0(\tau)(\sup(y, \alpha))$ to be
$\alpha(\tau, x)$. Unlike with $W_0$, this does not make $N_0$ into a
presheaf over $Z$. We will see why in the proof of lemma
\ref{lem:nfsarepresheaf}.



\subsubsection{The Presheaf of Natural Normal Forms}
\label{sec:presh-struct-norm}

By analogy with $W$ in section \ref{sec:dependent-w-types}, we define
a subobject $N$ of $N_0$. Given $\sup(y, \alpha) \in N_0(c, z)$, we
say it is \emph{natural} if for all $\sigma \colon c \rightarrow d$
and $\tau \colon d \rightarrow e$ in $\smcat{C}$ and all
$x \in X_{Y(\sigma)(y)}$, we have
$N_0(\tau)(\alpha(\sigma, x)) = \alpha(\tau \circ \sigma,
X(\tau)(x))$. We define the subobject $N$ of $N_0$ of
\emph{hereditarily natural} elements to be those of the form
$\sup(y, \alpha)$ which are natural and such that for all
$\sigma \colon c \rightarrow d$ and all $x \in X_{Y(\sigma)(y)}$,
$\alpha(\sigma, x)$ is hereditarily natural. Formally, we can define
this object using the same technique as for lemma \ref{lem:herednat}.

Note that for each $\tau \colon c \rightarrow c'$, $N_0(\tau)$
restricts to a map $N(c, z) \rightarrow N(c', Z(\tau)(z))$. We now
verify that this does give an internal presheaf.

\begin{lemma}
  \label{lem:nfsarepresheaf}
  $N$ with the action of morphisms defined above is a presheaf.
\end{lemma}

\begin{proof}
  It is straightforward to check that the action preserves identities.

  Now suppose we are given $\sigma \colon c \rightarrow d$ and
  $\tau \colon d \rightarrow e$. We need to verify that for all
  $v \in N(c, z)$, $N(\tau \circ \sigma)(v) =
  N(\tau)(N(\sigma)(v))$. We know that $v$ must be of the form
  $\sup(y, \alpha)$. The equation is straightforward to check when
  $Y(\sigma)(y)$ does not reduce. Hence we just show the case when
  $Y(\sigma)(y)$ reduces at $x \in X_y$, for which we will need
  naturality. Note that $Y(\tau \circ \sigma)(y)$ reduces at
  $X(\tau)(x)$.
  \begin{align*}
    N(\tau)(N(\sigma)(\sup(y, \alpha)))
    &= N(\tau)(\alpha(\sigma, x)) \\
    &= N_0(\tau)(\alpha(\sigma, x)) \\
    &= \alpha(\tau \circ \sigma, X(\tau)(x)) & \text{by naturality} \\
    &= N(\tau \circ \sigma)(\sup(y, \alpha))
  \end{align*}
\end{proof}

\subsubsection{The Algebra Structure}
\label{sec:algebra-structure}

It only remains to check that $N$ really is an initial algebra. In
this section we define the algebra structure $s$. We will use the
presentation we saw in section \ref{sec:form-intern-lang} where an
algebra structure is an algebra structure for the underlying
dependent polynomial endofunctor that satisfies the reduction
equations. We need to define $s_{z,c}(y, \alpha)$ whenever
$\alpha \colon f^\ast(\yoneda (c)) \rightarrow h^\ast(N)$. As
explained in section \ref{sec:dependent-w-types}, this is just an
element of
$\Pi_{\Sigma_{\sigma \colon c \rightarrow d} X(Y(\sigma)(y))} N(d,
Z(\sigma)(z))$ that satisfies the naturality condition. We split into
cases depending on whether $y$ reduces. If it does,
then we define $s(y, \alpha)$ to be $\alpha(x)$ where $y$ reduces at
$x$. Otherwise,
we take $s(y, \alpha)$ to be the element $\sup(y, \alpha)$ in $N_0$,
which in fact lies in $N$ since it is clearly hereditarily natural by
the fact that $\alpha$ maps into $N$ and is natural. We also need to
show that $s$ is natural, which we do in the lemma below.

\begin{lemma}
  The operation $s_{z,c}$ defined above is natural in the following
  sense. For any $\tau \colon c \rightarrow c'$ in $\smcat{C}$ and
  $z \in Z(c)$, we have the following commutative diagram (where the
  dependent product is the one internal in the category of
  presheaves).
  \begin{equation*}
    \xymatrix{ \Sigma_g \Pi_f h^\ast(N)(c) \ar[r]^{s_c} \ar[d] & N(c)
      \ar[d]^{N(\tau)} \\
      \Sigma_g \Pi_f h^\ast(N)(c') \ar[r]_{s_{c'}} & N(c')}
  \end{equation*}
\end{lemma}

\begin{proof}
  Let $(y, \alpha) \in \Sigma_g \Pi_f h^\ast(N)(c)$.

  There are three cases to consider. Either neither
  $y$ nor $Y(\tau)(y)$ reduces, or $Y(\tau)(y)$ reduces but not
  $y$, or $y$ reduces. The first case is essentially
  the same as for ordinary $W$-types in presheaves, and the other two
  cases are straightforward to check.
\end{proof}

Finally, we also need to check the reduction equations. However, note
that they hold internally if and only if they hold pointwise, and it
is clear that they do by the definition of $N$ and $s$.

We can now deduce the following lemma.
\begin{lemma}
  The operation $s_c$ defined above gives $N$ the structure of an
  algebra over the given pointed polynomial endofunctor.
\end{lemma}

\subsubsection{Proof of Initiality}
\label{sec:proof-initiality}

We now show that the algebra structure we have defined really is
initial. Suppose we are given an internal presheaf $A$ with the
structure of an algebra over the pointed polynomial endofunctor.  As
before we use the presentation in section \ref{sec:form-intern-lang},
where we view an algebra over the pointed endofunctor as
an algebra structure over the dependent polynomial endofunctor
$\Sigma_g \Pi_f h^\ast$, which we'll write as
$r \colon \Sigma_g \Pi_f h^\ast(A) \rightarrow A$, such that this
algebra structure satisfies the reduction equations. We need to define
a structure preserving map $t \colon N \rightarrow A$, and show that
it is the unique such map.

The basic idea for the definition of $t$ is fairly simple. Given
$\sup(y, \alpha)$ in $N(c, z)$, we want to define $t(\sup(y, \alpha))$
to be $r(y, t \circ \alpha)$. This is however quite tricky to
formalise, since $r(y, t \circ \alpha)$ is only well defined when we
know that $t \circ \alpha$ is natural, but this only makes sense when
we have already defined at least some of $t$. This issue already
occurs for
ordinary $W$-types in presheaves, but is especially
relevant here, where the proof of naturality is more
difficult. What we need to do is to simultaneously show that $t$ is
natural while we are defining it, since then we can deduce that
$t \circ \alpha$ is also natural, and so $r(y, t \circ \alpha)$ is
well defined.

To help us with this, we define another presheaf $T$, again using
dependent $W$-types in $\catc$ over $Z$, where we modify the
definition of $N$ by adding in also elements of $A$. We will in fact
construct $T$ in several stages, first using a dependent $W$-type,
$T_0$, then taking a succession of inductively defined subobjects
$T_1$, $T_2$ and finally $T$. In each case, we'll just give the
inductive definition, but in fact they can all be constructed in
arbitrary locally cartesian closed categories with $W$-types using
similar techniques to those in the proof of lemma \ref{lem:herednat}.

We first define the dependent $W$-type, $T_0$ by the following
inductive definition.

Let $c \in \smcat{C}$ and $z \in Z(c)$. Suppose that we are given
$y \in Y(c, z)$ such that $y$ does not reduce,
$a \in A(c, z)$ and $\alpha$ in
$\Pi_{\Sigma_{\sigma \colon c \rightarrow d} X(Y(\sigma)(y))} T_0(d,
Z(\sigma)(z))$. Then $T_0(c, z)$ contains an element of the form
$\sup(y, a, \alpha)$.

Note that we have a projection $\pi_0 \colon T_0 \rightarrow N_0$ over
$Z$ by simply ``forgetting'' the $a$'s. We also have a projection
$\pi_1 \colon T_0 \rightarrow A$ given by $\pi_1(\sup(y, a, \alpha))
:= a$.

We define $T_0(\tau) \colon T_0(c, z) \rightarrow T_0(c', Z(\tau)(z))$ the
same as for $N_0(\tau)$.  We now define $T_1$ to be the subobject of
$T_0$ of hereditarily natural elements, which is defined exactly the
same as in $N$. It follows that $\pi_0$ restricts to a function
$T_1 \rightarrow N$. We also have naturality in the following lemma.
\begin{lemma}
  \label{lem:proj0nat}
  Let $\tau \colon c \rightarrow c'$. Then $T_0(\tau)$ restricts to a
  morphism $T_1(\tau) \colon T_1(c, z) \rightarrow T_1(c',
  Z(\tau)(z))$. This makes $T_1$ into a presheaf, and the restriction
  of $\pi_0$ into a natural transformation.
\end{lemma}

\begin{proof}
  Since we mimicked the construction of $N$ from $N_0$, it's clear
  that we can use the same proof as in lemma \ref{lem:nfsarepresheaf}
  to show $T_1$ is a presheaf and that $\pi_0$ is natural.
\end{proof}

We now define a subobject $T_2$ of $T_1$ by the following inductive
definition. Given, $\sup(y, a, \alpha) \in T_1$, we say $\sup(y, a,
\alpha)$ belongs to $T_2$ if the following hold.
\begin{enumerate}
\item If $\sigma \colon c \rightarrow d$ is such that
  $Y(\sigma)(y)$ reduces at $x$, then $A(\sigma)(a) =
  \pi_1(\alpha(\sigma, x))$.
\item For all $\sigma \colon c \rightarrow d$ and $x \in X(d,
  Y(\tau)(y))$, $\alpha(\sigma, x) \in T_2$.
\end{enumerate}

We can now show the following lemma.
\begin{lemma}
  \label{lem:proj1nat}
  The restriction of $\pi_1$ to $T_2$ is natural.
\end{lemma}

\begin{proof}
  Suppose we are given $\sup(y, a, \alpha) \in T_2(c, z)$. We need to
  show that
  $\pi_1(T_2(\tau)(\sup(y, a, \alpha))) = A(\tau)(\pi_1(\sup(y, a,
  \alpha)))$. This is clear when $Y(\tau)(y)$ does not reduce.
  When $Y(\tau)(y)$ does reduce it's still
  clear, but we need to use the clause added to the definition of
  $T_2$ (it does not hold for $T_1$).
\end{proof}

The key point is that naturality in the definition of $T_1$ ensures
that we also have naturality for the composition of $\alpha$ with
projection to $A$, in the following sense.
\begin{lemma}
  For each $\sup(y, a, \alpha)$ in $T_2(c, z)$, $\pi_1 \circ \alpha$
  is natural.
\end{lemma}

\begin{proof}
  This is straightforward from the definition of $T_1$ (together with
  the observation that the same then applies when restricting to the
  subobject $T_2$) and lemma \ref{lem:proj1nat}.
\end{proof}

We now know that the expression $r(y, \pi_1 \circ \alpha)$ is well
defined, which finally allows us to define $T$ as the subobject of
$T_2$ defined inductively as follows. An element $\sup(y, a, \alpha)$
of $T_2$ belongs to $T$ if both of the conditions below hold.
\begin{enumerate}
\item $a = r(y, \pi_1 \circ \alpha)$
\item For all $\sigma \colon c \rightarrow d$ and $x \in X(d,
  Y(\tau)(y))$, $\alpha(\sigma, x) \in T$.
\end{enumerate}

We can now show the main lemma.
\begin{lemma}
  Let $T$ be as above. Then $\pi_0 \colon T \rightarrow N$ is an
  isomorphism.
\end{lemma}

\begin{proof}
  We show by induction on the construction of $N$ that for all $v \in
  N$, the fibre $\pi_0^{-1}(\{v\})$ in $T$ contains exactly one
  element.

  Suppose we are given an element of $N(c, z)$ of the form
  $\sup(y, \alpha)$. Clearly any element of
  $\pi_0^{-1}(\sup(y, \alpha))$ must be of the form
  $\sup(y, r(y, \pi_1 \circ \pi_0^{-1} \circ \alpha), \pi_0^{-1} \circ
  \alpha)$. We just need to check that this really is a well defined
  expression and that it belongs to $T$ (as opposed to just $T_0$,
  say).

  In the above, we were just using $\pi_0^{-1}$ as a convenient
  notation for a partial function, rather than a total inverse. Note
  however, that the induction hypothesis tells us that $\pi_0^{-1}
  \circ \alpha$ is a well defined function and the
  usual proof that the levelwise inverse of a
  natural transformation is natural still applies and,
  together with lemma \ref{lem:proj0nat} and the naturality of $\alpha$,
  allows us
  to show that $\pi_0^{-1}
  \circ \alpha$ is natural.

  It follows from the above together with lemma
  \ref{lem:proj1nat} that $\pi_1 \circ \pi_0^{-1} \circ \alpha$ is
  natural and so $r(y, \pi_1 \circ \pi_0^{-1} \circ \alpha)$
  is a well defined expression. Hence
  $\sup(y, r(y, \pi_1 \circ \pi_0^{-1} \circ \alpha), \pi_0^{-1} \circ
  \alpha)$ is a valid expression for an element of $T_0$. We just need
  to show that it belongs to the subobject $T$.

  From the naturality of $\pi_0^{-1} \circ \alpha$ that we've already
  seen, it's clear that
  $\sup(y, r(y, \pi_1 \circ \pi_0^{-1} \circ \alpha), \pi_0^{-1} \circ
  \alpha)$ belongs to $T_1$.

  To show it belongs to $T_2$, we need to show that when
  $\tau \colon c \rightarrow d$ is such that
  $Y(\tau)(y)$ reduces at $x$, we have
  $A(\tau)(r(y, \pi_1 \circ \pi_0^{-1} \circ \alpha)) =
  \pi_1(\pi_0^{-1}(\alpha(\tau, x)))$. However, this follows directly
  from the naturality of $r$ together with the fact that $r$ was
  required to
  respect the reduction equations.

  It's now clear that
  $\sup(y, r(y, \pi_1 \circ \pi_0^{-1} \circ \alpha), \pi_0^{-1} \circ
  \alpha)$ belongs to $\pi_0^{-1}(\{\sup(y, \alpha)\})$ in $T$ and that in
  fact it's the unique such object.
\end{proof}

We can now define $t \colon N \rightarrow A$ to be $\pi_1 \circ
\pi_0^{-1}$. We now just need to check that it is a structure
preserving map, and unique with this property.

\begin{lemma}
  The map $t \colon N \rightarrow A$ defined by $\pi_1 \circ
  \pi_0^{-1}$ is a natural transformation that is structure preserving
  and is the unique such map.
\end{lemma}

\begin{proof}
  Naturality follows from lemmas \ref{lem:proj0nat} and
  \ref{lem:proj1nat}.

  To show that $t$ is structure preserving, we again need to split
  into two cases depending on whether there is a reduction. However,
  both cases are straightforward to show from the definition.

  It's also clear from the definition that $t$ is the unique structure
  preserving map, and in fact for uniqueness it's sufficient just to
  look at the case where there is no reduction.
\end{proof}

We can now deduce the main theorem of this section.
\begin{theorem}
  \label{thm:initialalgps}
  In any category of internal presheaves in a locally cartesian closed
  category with disjoint coproducts, every locally decidable
  pointed polynomial endofunctor has an initial
  algebra.
\end{theorem}

\section{$W$-Types with Reductions in Classical Logic}
\label{sec:w-types-with-1}

We will see in this section how to construct all $W$-types with
reductions in boolean toposes with natural number object. We have
already seen the main idea in the previous section. Every topos is a
category of internal presheaves over itself via the trivial category,
and in this case locally decidable is the same as decidable. For a
boolean topos, a polynomial with reductions is decidable just when the
map $f \circ k$ is monic. This only leaves the case where $f \circ k$
is not monic. What this says is that the same constructor can reduce
in more than one place. The key point is that when we know that this
happens, things become trivial, in the following sense.

\begin{lemma}
  \label{lem:classicalcollapses}
  Suppose we are given a polynomial with reductions of the form
  \eqref{eq:27}. Let $(A_z)_{z \in Z}$ be a family of types over $Z$
  with algebra structure given by $c$ (which we will view as an
  algebra on the underlying polynomial that satisfies the reduction
  equations). Suppose that for some $z \in Z$ there is a constructor
  $y \in Y_z$ that reduces in two distinct places $x_1 \neq x_2 \in
  X_y$ and there exists a dependent function $\alpha : \Pi_{x \in X_y}
  A_{h(z)}$. Then $A_z$ contains exactly one element.
\end{lemma}

\begin{proof}
  First of all, note that $A_z$ contains at least one element using
  the algebra structure, which is $c(y, \alpha)$.

  Next, suppose that $a_1$ and $a_2$ are both elements of $A_z$. Then
  we define a new dependent function $\alpha'$ as follows.
  \begin{equation*}
    \alpha'(x) :=
    \begin{cases}
      a_1 & x = x_1 \\
      a_2 & x = x_2 \\
      \alpha(x) & \text{otherwise}
    \end{cases}
  \end{equation*}
  Note that the coherence condition ensures that this is still a
  dependent function of type $\Pi_{x : X_y} A_{h(x)}$. Also note that
  we needed classical logic to show this is a well defined function.

  Then the reduction equation at $x_1$ tell us $c(y, \alpha') = a_1$,
  and the reduction equation at $x_2$ tells us $c(y, \alpha') =
  a_2$. Hence $a_1 = a_2$. Therefore, $A_z$ contains exactly one
  element.
\end{proof}

We will now use this idea to construct any $W$-type with
reductions. We aim towards the following theorem.
\begin{theorem}
  \label{thm:classicalwr}
  Let $\catc$ be a boolean topos with natural number object. Then
  $\catc$ has all $W$-types with reductions.
\end{theorem}

We first define a useful construction.  Suppose we are
given a subobject $C \subseteq Z$. Then we construct a new polynomial
as follows. We work over the same context $Z$. For $z \in C$, we
define the set of constructors $Y'_z$ to consist of exactly one
element $\ast$, with empty arity $X'_\ast := \emptyset$.

Otherwise, for $z \notin C$, we define $Y'_z$ to be the subobject of
$Y_z$ consisting of those $y$ with no reductions. That is, those where
$R_{y, x} = \emptyset$ for all $x \in X_y$. We define the arity $X'_y$
to be $X_y$.

Write $W^C$ for the resulting $W$-type on the polynomial. Observe that
for $z \in C$, $W^C_z$ has exactly one element, of the form
$\sup(\ast, \emptyset)$, where $\ast$ is the only constructor over
$z$.

\begin{remark}
  For the special case $C = \emptyset$, this gives us the definition
  of normal forms like in section \ref{sec:norm-forms-norm}. For the
  special case $C = Z$, the resulting $W$-type contains exactly one
  element in every fibre of $z \in Z$.
\end{remark}

We say that $C$ is \emph{closed} if whenever $z \in Z$ is such that
there exists a constructor $y \in Y_z$ that reduces in two distinct
places $x_1 \neq x_2$ and there exists some dependent function
$\alpha : \Pi_{x : X_y} W^C_{h(x)}$, then we have $z \in C$.

We then define $C_0$ to be the intersection of all closed sets
$C$.

\begin{lemma}
  \label{lem:c0isclosed}
  $C_0$ is itself closed.
\end{lemma}

\begin{proof}
  Let $z \in Z$ be such that there exists a constructor
  $y \in Y_z$ that reduces in two distinct places $x_1 \neq x_2$ and
  let $\alpha : \Pi_{x : X_y} W^{C_0}_{h(x)}$. We need to show that
  for any closed set $C$, $z \in C$, so let $C$ be an arbitrary closed
  set.

  We first construct a map $i \colon W^{C_0} \rightarrow W^C$ over $Z$
  recursively as follows. Suppose that $z' \in Z$, and we are given an
  element of $W^{C_0}_{z'}$ of the form $\sup(y, \alpha)$.

  First suppose that $z' \in C$. In this case we take $i(\sup(y,
  \alpha))$ to be the unique element of $W^C_{z'}$.

  Otherwise we know that $z' \notin C$. In that case, we define
  $i(\sup(y, \alpha))$ to be $\sup(y, i \circ \alpha)$, which is a
  valid element of $W^C_{z'}$ since $z' \notin C$, and also $z' \notin
  C_0$ (since $C_0 \subseteq C$).
  
  We then use $i$ to construct an element of
  $\Pi_{x \in X_y} W^C_{h(x)}$ defined by $i \circ \alpha$. But we can
  now deduce that $z \in C$.

  Since we showed $z \in C$ for any closed set, we have $z \in C_0$,
  and so $C_0$ is closed, as required.
\end{proof}

\begin{lemma}
  \label{lem:classicalalgstr}
  For any closed set $C$, we give $W^{C}$ an algebra structure $d$ for
  our given polynomial with reductions.
\end{lemma}

\begin{proof}
  Suppose we are given $y \in Y_z$ for some $z \in Z$, and a dependent
  function $\alpha : \Pi_{x \in X_y} W^{C}_{h(x)}$. To define
  $d(y, \alpha)$ we split into cases. Firstly, if $z \in C$, we take
  $d(y, \alpha)$ to be the unique element of $W^C_z$. Now consider
  just the case when $z \notin C$. If $y$ reduces in two different
  places, then we could show $z \in C$, since $C$ is closed, deriving
  a contradiction. Hence we may assume that $y$ either reduces exactly
  once, or not at all. We now proceed the same as in section
  \ref{sec:algebra-structure}. If $y$ reduces at $x$, we define
  $d(y, \alpha)$ to be $\alpha(x)$. Otherwise $y$ does not reduce at
  all, and so we can use the $W$-type structure and take
  $d(y, \alpha)$ to be $\sup(y, \alpha)$.

  This algebra structure clearly satisfies the reduction equations.
\end{proof}

\begin{lemma}
  $W^{C_0}$ with the algebra structure given in lemma
  \ref{lem:classicalalgstr} is initial.
\end{lemma}

\begin{proof}
  Suppose we are given a family of types $(A_z)_{z \in Z}$ with
  algebra structure $c$. We need to show that there is a unique
  structure preserving map $i \colon W^{C_0} \rightarrow A$ over $Z$.

  We define $C$ to consist of those $z \in Z$ such that $A_z$ contains
  exactly one element. We now recursively define a map
  $j \colon W^C \rightarrow A$. Suppose we are given $z \in Z$, and
  $\sup(y, \alpha) \in W^C_z$. If $y = \ast$, then we must have
  $z \in C$. But then we can take $j(\sup(\ast, \alpha))$ to be the
  unique element of $A_z$. Otherwise, $y$ must be one of the original
  constructors in $Y_z$, and $\alpha : \Pi_{x : X_y} W^C_{h(x)}$. We
  define $j(\sup(y, \alpha))$ to be $c(y, j \circ \alpha)$.

  We can now deduce that $C$ is closed, since if we are given a
  constructor $y \in Y_z$ that reduces in two distinct places and a
  dependent function $\alpha : \Pi_{x : X_y} W^C_{h(x)}$, then by
  considering $j \circ \alpha$, we show by lemma
  \ref{lem:classicalcollapses} that $A_z$ has exactly one element, and
  so $z \in C$. But this implies that $C_0 \subseteq C$, and so we get
  a canonical map $W^{C_0} \rightarrow W^C$, as in the proof of lemma
  \ref{lem:c0isclosed}. Composing with $j$ gives us the map $W^{C_0}
  \rightarrow A$ over $Z$.

  However, it is now straightforward to check that this is the unique
  structure preserving map.
\end{proof}

We can now use the above lemma to deduce the main theorem
\ref{thm:classicalwr}.

\section{Cofibrantly Generated Awfs's in Codomain Fibrations}
\label{sec:cofibr-gener-awfss}

\subsection{Review of Lifting Problems over Codomain Fibrations}
\label{sec:liftprobreview1}

We recall some definitions from \cite[Section
7.5]{swanliftprob}\todo{maintain ref to other paper}.
Since we focus only on the special case of
codomain fibrations, we can simplify some of the definitions a little.

\begin{definition}
  \label{def:liftprob1}
  Let $f$ be a map in $\catc/I$ and let $g$ be a map in $\catc/J$. A
  \emph{family of lifting problems} from $f$ to $g$ over $K \in \catc$
  is diagram of the following form, where the squares on the left are
  both pullbacks.
  \begin{equation*}
    \begin{gathered}
      \xymatrix{ U \ar[d] & \sigma^\ast(U) \ar[l] \ar[r] \ar[d]
        \pullbackcorner[dl] & X \ar[d] \\
        V \ar[d] & \sigma^\ast(V) \ar[l] \ar[r] \ar[d]
        \pullbackcorner[dl] & Y \ar[d] \\
        I & K \ar[l]_\sigma \ar[r] & J
      }
    \end{gathered}
  \end{equation*}
  A \emph{solution} to the family of lifting problems is a map
  $\sigma^\ast(V) \rightarrow X$ making the upper right square into
  two commutative triangles.
\end{definition}

\begin{definition}
  \label{def:ulp}
  Let $f$ be a map in $\catc/I$ and let $g$ be a map in $\catc/J$. The
  \emph{universal family of lifting problems} from $f$ to $g$, is the
  family of lifting problems, where we define $K$ to be type below,
  \begin{equation*}
    \Sigma_{i : I} \Sigma_{j : J} \Sigma_{\beta : V(i) \rightarrow Y(j)}
    \Pi_{v : V(i)} (U(i, v)
    \rightarrow X(j, \beta(v)))
  \end{equation*}
  and the right maps in the family of lifting problems are given by
  evaluation.
\end{definition}

\begin{definition}
  Fix a map $Y \rightarrow J$. \emph{Step 1 of the small object
    argument at $Y$} is the pointed endofunctor
  $R_1 \colon \catc/Y \rightarrow \catc/Y$ defined as follows. Suppose
  that we are given $f \colon X \rightarrow Y$ in $\catc/Y$. We first
  form the universal lifting problem from $m$ to $f$ as in definition
  \ref{def:ulp}. We then define $R_1 f$ to be the unique map out of
  the pushout, with unit given by the pushout inclusion $\lambda_f$,
  as below.
  \begin{equation*}
    \xymatrix{ \sigma^\ast(U) \ar[rr] \ar[d] & & X \ar[dd]
      \ar[dl]|{\lambda_f} \\
      \sigma^\ast(V) \ar[r] \ar@/_/[drr]
      & K_1 f \pushoutcorner \ar[dr]|{R_1 f} & \\
      & & Y}
  \end{equation*}
\end{definition}

We recall the following from \cite[Theorem 7.5.2]{swanliftprob} (see
also \cite[Remark 7.5.6]{swanliftprob}, and \cite[Section
4.4]{swanliftprob} for the more general and precise definitions of
fibred and strongly fibred).
\begin{proposition}
  The pointed endofunctors are preserved by pullback along all maps
  $J' \rightarrow J$. We say $R_1$ is a \emph{fibred} lawfs.
\end{proposition}

\begin{definition}
  We say $R_1$ is \emph{strongly fibred} if it is preserved by
  pullback along all maps $Y' \rightarrow Y$.
\end{definition}

Given any $f \colon X \rightarrow Y$ in $\catc/Y$, we have a pointed
endofunctor, which we will denote $I_{X}$, defined by coproduct,
sending $X'$ to $X' + X$, with unit given by coproduct inclusion. We
clearly have the following proposition (by taking reductions and
arities both to be initial).
\begin{proposition}
  For any $X$, $I_X$ is pointed polynomial.
\end{proposition}

\begin{theorem}
  \label{thm:cofgenawfsfrominit}
  Suppose that for each map $Y \rightarrow J$ and every $f \colon X
  \rightarrow Y$ in $\catc/Y$ we are given a choice of initial algebra
  for the pointed endofunctor $I_X + R_1$. Then the awfs cofibrantly
  generated by $m$ exists, and is fibred.
\end{theorem}

\begin{proof}
  See \cite[Corollary 5.4.7]{swanliftprob}.
  \todo{maintain ref to other paper}
\end{proof}

\begin{theorem}
  If $R_1$ is strongly fibred
  then so is the resulting cofibrantly generated rawfs, if it exists.
\end{theorem}

\begin{proof}
  See \cite[Theorem 5.5.2]{swanliftprob}.
\end{proof}

\subsection{Step 1 as a Pointed Polynomial Endofunctor}
\label{sec:step-1-as}

\begin{theorem}
  \label{thm:step1isptdpoly}
  $R_1$ is pointed polynomial.
\end{theorem}

\begin{proof}
  Unfolding the type theoretic definition of universal lifting
  problem, we get the following descriptions of $\sigma^\ast(U)$ and
  $\sigma^\ast(V)$.
  \begin{align*}
    \sigma^\ast(U) &\cong \Sigma_{j : J} \Sigma_{i : I} \Sigma_{v_0 :
                     V(i)} \Sigma_{\beta: V(i) \rightarrow Y(j)}
                     \Sigma_{u : U(i, v_0)}
                     \Pi_{z : \Sigma_{v : V(i)} U(i, v)}X(j,
                     \beta(p_0(z))) \\
    \sigma^\ast(V) &\cong \Sigma_{j : J} \Sigma_{i : I} \Sigma_{v_0 :
                     V(i)} \Sigma_{\beta: V(i) \rightarrow Y(j)}
                     \Pi_{z : \Sigma_{v : V(i)} U(i, v)}X(j,
                     \beta(p_0(z)))
  \end{align*}
  However, like this it is clear that the definition matches the
  definition of pointed polynomial endofunctor.
\end{proof}

It is easiest to understand the definition of the polynomial with
reductions for $R_1$ when we phrase it in terms of constructors,
arities, reindexing and reductions. We read these off from the
description above.

The overall context we are working in is the object $Y$, which in type
theoretic notation is $\Sigma_{j : J} Y(j)$ (since we are thinking of
$Y$ as a family of types indexed by $J$).

A constructor over $(j, y)$ for $j : J$ and $y : Y(j)$ consists of
$i : I$, $v_0 : V(i)$ and a map $\beta \colon V(i) \rightarrow Y(j)$
such that $\beta(v_0) = y$.

The arity of the constructor $(i, v_0, \beta)$ is $\Sigma_{v: V(i)}
U(i, v)$.

The reindexing map sends $(i, v_0, \beta, (v, u))$ to $\beta(v)$.

Finally, the reduction equations say that given
$\alpha : \Pi_{\Sigma_{v : V(i)} U(i, v)} X(j, \beta(j, p_0(z)))$ and
$u : U(v_0)$, $\sup(i, v_0, \beta, \alpha)$ reduces to $\alpha(v_0,
u_0)$ (where $\sup(i, v_0, \beta, \alpha)$ is given by some
$R_1$-algebra structure).

We can think of the corresponding $W$-type with reductions directly in
terms of lifting problems as follows. Suppose we are given a
constructor $(i, v_0, \beta)$ and a map
$\alpha : \Pi_{\Sigma_{v : V(i)} U(i, v)} X(j, \beta(j,
p_0(z)))$. Then, firstly $\beta$ and $\alpha$ together form a lifting
problem of $m_i$ against $f_j$. We think of
$\sup(i, v_0, \beta, \alpha)$ as a diagonal filler of the lifting
problem, evaluated at $v_0$. The reduction equations then ensure that
the upper triangle of the diagonal filler commutes. Therefore, we
think of an initial algebra of $R_1$ as the result of freely adding a
filler for every lifting problem, subject to ensuring that the upper
triangles do always commute.

An initial algebra for $R_1 + I_X$ is similar. Once again, we are
freely adding a filler for every lifting problem. However in this case
we start off with a copy of $X$ before adding all the fillers.

Finally, we will later need the lemma below.
\begin{lemma}
  \label{lem:step1ispb}
  For each $Y \rightarrow J$, $R_1$ at $Y$ is generated by the polynomial
  with reductions of the form below, where the map $A \rightarrow C$
  is a pullback of the map $U \rightarrow I$.
  \begin{equation*}
    \begin{gathered}
      \xymatrix { R \ar[r] & A \ar[r] \ar[dl] & C \ar[dr] \\
        Y & & & Y}
    \end{gathered}
  \end{equation*}
\end{lemma}

\begin{proof}
  We can read off an description of the map $A \rightarrow C$ from the
  arguments above.\footnote{The same is true for the other maps, but we don't
  need them here, and it is somewhat messy.}

  In type theoretic notation, $A$ and $C$ are defined as below, with
  the map $A \rightarrow C$ given by projection.
  \begin{align*}
    C &:= \Sigma_{j : J} \Sigma_{i : I} \Sigma_{v_0 :
                     V(i)} \Sigma_{\beta: V(i) \rightarrow Y(j)} \\
    A &:= \Sigma_{j : J} \Sigma_{i : I} \Sigma_{v_0 :
                     V(i)} \Sigma_{\beta: V(i) \rightarrow Y(j)}
                     \Sigma_{v : V(j)} U(i, v)
  \end{align*}

  However, in this form it is clear that the map $A \rightarrow C$ is
  just the pullback of the map $U \rightarrow I$ along the projection
  $C \rightarrow I$.
\end{proof}

We can now deduce the following.
\begin{theorem}
  \label{thm:twocovpretoposcofibgen}
  Suppose we are given a family of maps of the following form over the
  codomain functor on a $\Pi W$-pretopos
  \begin{equation*}
    \xymatrix{ U \ar[dr] \ar[rr]^m & & V \ar[dl] \\
      & I & }
  \end{equation*}
  Furthermore suppose we are given a 2-cover base of the map
  $U \rightarrow I$. Then $m$ cofibrantly generates an awfs.
\end{theorem}

\begin{proof}
  We have shown in theorem \ref{thm:step1isptdpoly} that $R_1$ is
  pointed polynomial. Hence for each $f \colon X \rightarrow Y$, the
  pointed endofunctor $R_1 + I_X$ from theorem
  \ref{thm:cofgenawfsfrominit} is also pointed polynomial.

  $I_X$ trivially has a $2$-cover base. $R_1$ has a $2$-cover base since
  by lemma \ref{lem:step1ispb} it is a pullback of the map
  $U \rightarrow I$ for which we are given a $2$-cover base and so we
  can apply lemma \ref{lem:twocovpb}.
  
  Hence we can construct a $2$-cover base for each $R_1 + I_X$ by
  lemma \ref{lem:twocovcoprod}.
  
  But then by theorem \ref{thm:rwinitialalg} we can find initial
  algebras, so we can deduce by theorem \ref{thm:cofgenawfsfrominit}
  that the cofibrantly generated awfs on $m$ exists.
\end{proof}

\begin{corollary}
  \label{cor:pwtoposcofibgen}
  Suppose we are given a family of maps of the following form over the
  codomain functor on a $\Pi W$-pretopos satisfying $\wisc$
  \begin{equation*}
    \xymatrix{ U \ar[dr] \ar[rr]^m & & V \ar[dl] \\
      & I & }
  \end{equation*}
  Then the awfs cofibrantly generated by the family of maps exists.
\end{corollary}

\begin{proof}
  By $\wisc$, the map $U \rightarrow I$ has a 2-cover base. Hence we can
  apply theorem \ref{thm:twocovpretoposcofibgen}.
\end{proof}

\begin{remark}
  One might expect that corollary \ref{cor:pwtoposcofibgen} can be
  proved directly without going via theorem
  \ref{thm:twocovpretoposcofibgen}, by using $\wisc$ directly to find
  each $2$-cover base. However, this doesn't work because we need to
  have a choice of $2$-cover bases for every vertical map
  $X \rightarrow Y \rightarrow J$, and $\wisc$ only tells us at least
  one such $2$-cover base exists. When we use theorem
  \ref{thm:twocovpretoposcofibgen} this does not matter because we
  only have to apply $\wisc$ once (or rather, twice), to get a
  $2$-cover base for the map $U
  \rightarrow I$, and from that we can define all the other
  $2$-cover bases that we need.
\end{remark}

We can also apply the simplified construction from section
\ref{sec:simpl-categ-presh} to get the following theorem.

\begin{theorem}
  \label{thm:simplecofgen}
  Let $\catc$ be a finitely cocomplete locally cartesian
  closed category with disjoint coproducts.
  Let $\smcat{A}$ be an internal category in $\catc$,
  and $\catc^\smcat{A}$ the category of diagrams of shape $\smcat{A}$.
  Suppose we are given a family of maps of the
  following form over the codomain functor on $\catc^\smcat{A}$.
  \begin{equation*}
    \xymatrix{ U \ar[dr] \ar[rr]^m & & V \ar[dl] \\
      & I & }
  \end{equation*}
  Suppose further that the map $U \rightarrow V$ is locally decidable.

  Then the awfs cofibrantly generated by the diagram exists.
\end{theorem}

\begin{proof}
  Note that locally decidable maps are closed under pullback and
  coproduct. Hence, by a similar argument to the one in the proof of
  theorem \ref{thm:twocovpretoposcofibgen}, we see that $R_1 + I_X$ is
  a locally decidable point polynomial endofunctor. We can then deduce
  the result by theorems \ref{thm:cofgenawfsfrominit} and
  \ref{thm:initialalgps}.
\end{proof}

\subsection{Lifting Problems for Squares}
\label{sec:lift-probl-squar}

Recall that in \cite[Section 8]{swanliftprob} \todo{maintain citations
  to other paper}
the author showed
that Sattler's notion of lifting problem for squares (from
\cite{sattlermodelstructures}) can be generalised to work over a
fibration. We apply this to the codomain fibration on $\catc$ to get
the following.

Suppose that we are given a diagram of the following from.
\begin{equation}
  \label{eq:codsquare}
  \begin{gathered}
    \xymatrix{ U_0 \ar[rr] \ar[d]_{m_0} & & U_1 \ar[d]^{m_1} \\
      V_0 \ar[rr] \ar[dr] & & V_1 \ar[dl] \\
      & I &
    }    
  \end{gathered}
\end{equation}

\begin{definition}
  Let $f \colon X \rightarrow Y$ be a morphism in $\catc/J$ for some
  $J \in \catc$. We say a
  \emph{family of lifting problems} from \eqref{eq:codsquare} to $f$
  is a family of lifting problems (in the sense of definition
  \ref{def:liftprob1}) from $m_1$ to $f$.
\end{definition}

Note that pasting the family of lifting problems to the pullback of
\eqref{eq:codsquare} gives a commutative diagram of the following form.
\begin{equation}
  \label{eq:lpsqdiag}
  \begin{gathered}
    \xymatrix{ \sigma^\ast(U_0) \ar[d] \ar[r] & \sigma^\ast(U_1)
      \ar[r] \ar[d]
       & X \ar[d] \\
      \sigma^\ast(V_0) \ar[r] \ar[dr] & \sigma^\ast(V_1) \ar[r] \ar[d]
       & Y \ar[d] \\
       & K \ar[r] & J
    }
  \end{gathered}
\end{equation}

\begin{definition}
  A \emph{solution} to the family of lifting problems is a map
  $\sigma^\ast(V_0) \rightarrow X$ making the upper rectangle in
  \eqref{eq:lpsqdiag} into two commutative triangles.
\end{definition}

\begin{definition}
  The \emph{universal family of lifting problems} from
  \eqref{eq:codsquare} to $f \colon X \rightarrow Y$ is the universal
  family of lifting problems from $m_1$ to $f$.
\end{definition}

Recall from section \ref{sec:liftprobreview1} that the universal
lifting problem is defined type theoretically by taking $K$ to be the
following type, with the right maps given by evaluation.
\begin{equation*}
  \Sigma_{i : I} \Sigma_{j : J} \Sigma_{\beta : V_1(i) \rightarrow Y(j)}
  \Pi_{v : V_1(i)} (U_1(i, v)
  \rightarrow X(j, \beta(v)))
\end{equation*}

We use this to construct a pointed endofunctor over $\cod$.
\begin{definition}
  Fix a square over an object $I$ as in \eqref{eq:codsquare}. We
  define a pointed endofunctor $R_1$ over $\cod$ called \emph{step one
    of the small object argument} as follows. Given
  $f \colon X \rightarrow Y$ we define $R_1 f$ to be the map below
  given by the universal property of the pushout, where we take
  $\sigma \colon K \rightarrow I$ to be as in the universal lifting
  problem from the square to $f$. The unit at $f$, is given by the
  inclusion $\lambda_f$ into the pushout.
  \begin{equation*}
    \xymatrix{ \sigma^\ast(U_0) \ar[d] \ar[r] & \sigma^\ast(U_1) \ar[rr] & &
      X \ar[dd] \ar[dl]|{\lambda_f} \\
      \sigma^\ast(V_0) \ar[rr] \ar@/_/[drrr] & \ar[r] 
      & K_1 f \pushoutcorner \ar[dr]|{R_1 f} & \\
      & & & Y}
  \end{equation*}
\end{definition}

\begin{lemma}
  \label{lem:step1ispolysq}
  For any family of squares as in \eqref{eq:codsquare}, step one of
  the small object argument is a pointed polynomial endofunctor.
\end{lemma}

\begin{proof}
  By unfolding the type theoretic definition, similarly to as in
  theorem \ref{thm:step1isptdpoly}.
\end{proof}

\begin{theorem}
  Suppose that $\catc$ is a locally cartesian closed category and we
  are given a family of squares as in \eqref{eq:codsquare}. Suppose
  further that one of the following two conditions holds.
  \begin{enumerate}
  \item $\catc$ is a $\Pi W$-pretopos that satisfies $\wisc$.
  \item $\catc$ is a category of internal presheaves over a finitely
    cocomplete
    locally cartesian closed category with disjoint coproducts, and
    the map $U_0 \rightarrow V_0$ is a locally decidable
    monomorphism.
  \end{enumerate}
  Then the rawfs cofibrantly generated by \eqref{eq:codsquare}
  exists.

  Furthermore, if the map $V_1 \rightarrow I$ is an isomorphism then
  the resulting rawfs is strongly fibred.
\end{theorem}

\begin{proof}
  Similar to corollary \ref{cor:pwtoposcofibgen} and theorem
  \ref{thm:simplecofgen}, this time using lemma
  \ref{lem:step1ispolysq} and \cite[Theorem 5.3.6]{swanliftprob}.

  For showing the rawfs is strongly fibred, we use \cite[Theorem 5.3.8
  and Lemma 8.2.1]{swanliftprob}.
\end{proof}

\section{Recovering $W$-Types from Cofibrantly Generated Awfs's}
\label{sec:non-existence-due}

In section \ref{sec:cofibr-gener-awfss} we saw that cofibrantly generated
awfs's could be constructed using $W$-types and $\wisc$. We will know
show that the assumption of the existence of $W$-types is strictly
necessary. We will show that in fact $W$-types can be
recovered from the existence of cofibrantly generated awfs's. This shows
that the results in section \ref{sec:cofibr-gener-awfss} don't hold
for the category of sets in $\czf$, even if we add $\mathbf{PAx}$, a
choice axiom which implies $\wisc$.

\begin{theorem}
  \label{thm:awfs2wtype}
  Let $\catc$ be a locally cartesian closed category with disjoint
  coproducts.
  Suppose that every monic decidable family of maps cofibrantly generates an
  awfs. Then $\catc$ has all $W$-types.
\end{theorem}

\begin{proof}
  Let $f \colon A \rightarrow B$ be a morphism in $\catc$. We then
  consider the following family of morphisms.
  \begin{equation*}
    \begin{gathered}
      \xymatrix{ A \ar[dr]_f \ar@{>->}[rr]^{\iota_0} & & A + B
        \ar[dl]^{[f, 1_B]} \\
        & B & }
    \end{gathered}
  \end{equation*}
  Let $X$ be an object of $\catc$. Writing the local exponential as a
  dependent product, the universal lifting problem of $\iota_0$
  against the unique map $X \rightarrow 1$ is of the following form,
  where the top map is given by evaluation.
  \begin{equation*}
    \begin{gathered}
      \xymatrix{ \Pi_f(A^\ast(X)) \times_B A \ar[r] \ar[d] & X \ar[d] \\
        \Pi_f(A^\ast(X)) \times_B (A + B) \ar[r] & 1}
    \end{gathered}
  \end{equation*}
  Since $\catc$ is locally cartesian closed, pullback preserves
  coproduct, and so we have $\Pi_f(A^\ast(X)) \times_B (A + B) \cong
  (\Pi_f(A^\ast(X)) \times_B A) + (\Pi_f(A^\ast(X)) \times_B B)$. The
  second component is just the pullback of an identity map, so we
  deduce that the universal lifting problem is actually of the form
  below.
  \begin{equation*}
    \begin{gathered}
      \xymatrix{ \Pi_f(A^\ast(X)) \times_B A \ar[r] \ar[d] & X \ar[d] \\
        (\Pi_f(A^\ast(X)) \times_B A) + \Pi_f(A^\ast(X))) \ar[r] & 1}
    \end{gathered}
  \end{equation*}
  We deduce that solutions to the universal lifting problem correspond
  precisely to algebra structures on $X$ for the polynomial
  endofunctor $\Sigma_B \circ \Pi_f \circ A^\ast$. Therefore an
  initial algebra for the polynomial endofunctor is exactly the
  factorisation of $0 \rightarrow 1$ in the cofibrantly generated awfs.
\end{proof}

\begin{corollary}
  In $\czf + \mathbf{PAx}$ one cannot prove that cofibrantly generated
  awfs's exist for every monic decidable family of maps for the
  codomain fibration over the category of sets.
\end{corollary}

\begin{proof}
  By theorem \ref{thm:awfs2wtype} the existence of cofibrantly generated
  awfs's implies that the category of sets has $W$-types. However,
  $\czf + \mathbf{PAx}$ has the same consistency strength as $\czf$
  itself, but the addition of $W$-types leads to a strictly higher
  consistency strength (see \cite{Rathjen94}).
\end{proof}

\section{Examples of Previously Unknown Awfs's}
\label{sec:exampl-prev-unkn}

We now give some new examples of awfs's, all based on
realizability. We assume that the reader is already familiar with well
known definitions in realizability such as pca's, assemblies and
realizability and relative realizability toposes. See the reference
\cite{vanoosten} by Van Oosten for a comprehensive introduction to all
of these notions. We will use the same terminology and notation
as Van Oosten.

None of these categories admit colimits over arbitrary infinite
sequences (even countably infinite sequences).

\subsection{Kan Fibrations in the Effective Topos}
\label{sec:effective-topos}

In \cite{vdbergfrumin}, Van den Berg and Frumin considered two classes
of maps in the effective topos, $\eff$ referred to as \emph{trivial
  fibrations} and \emph{fibrations}. In \cite[Section
7.5.2]{swanliftprob}, the author showed that these classes are both
cofibrantly generated with respect to the codomain fibration, by the
following two families of maps.
\begin{equation*}
  \begin{gathered}
    \xymatrix{ 1 \ar[dr]_\top \ar[rr]^\top & & \Omega \ar@{=}[dl] \\
      & \Omega &}
  \end{gathered}
  \qquad
  \begin{gathered}
    \xymatrix{ \Omega +_1 (\Omega \times \nabla 2) \ar[dr] \ar[rr]^{\top \hat
        \times \delta_0}     
      & & \Omega \times \nabla 2
      \ar[dl]^{\pi_0} \\
      & \Omega &}      
  \end{gathered}  
\end{equation*}

In loc. cit., Van den Berg and Frumin showed that if one restricts to
the full subcategory of $\eff$ of \emph{fibrant objects} (i.e. objects
$X$ where the unique map $X \rightarrow 1$ is a fibration) then
fibrations are the right classes of a
wfs, and moreover this forms part of a model structure on the
subcategory. However, their proof relies on restricting to fibrant
objects, and doesn't apply to the entire category $\eff$.

We can now confirm that in fact, we do get awfs's on all of $\eff$,
without restricting to fibrant objects.
\begin{theorem}
  There are awfs's $(C, F^t)$ and $(C^t, F)$ on $\eff$ such that
  \begin{enumerate}
  \item A map admits an $F^t$-algebra structure if and only if it is a
    trivial fibration.
  \item A map admits an $F$-algebra structure if and only if it is a
    fibration.
  \item The awfs $(C, F^t)$ is strongly fibred (i.e. stable under
    pullback).
  \end{enumerate}
\end{theorem}

\begin{proof}
  In \cite{vdbergpredtop}, Van den Berg showed that $\eff$ is a $\Pi
  W$-pretopos and satisfies $\wisc$ (there referred to as AMC). We can
  therefore construct $(C, F^t)$ and $(C^t, F)$ using corollary
  \ref{cor:pwtoposcofibgen}. We see that $(C, F^t)$ is strongly fibred
  by \cite[Corollary 7.5.5]{swanliftprob}.\todo{maintain citation to
    other paper}
\end{proof}

\begin{remark}
  In fact we can define $(C, F^t)$ in two different ways. We can
  either take the underlying lawfs to be $(C_1, F^t_1)$ together with
  a multiplication that we can add using the fact that cofibrations
  can be composed. Alternatively, we can take $(C, F^t)$ to be the
  awfs algebraically free on $(C_1, F^t_1)$.  As Gambino and Sattler
  point out in \cite[Remark 9.5]{gambinosattlerpi} these two
  definitions are not
  the same. However, both are strongly fibred
  and we end up with the same wfs in either case.
\end{remark}

\subsection{Computable Hurewicz Fibrations in the Kleene-Vesley Topos}
\label{sec:funct-real}

Recall that the \emph{function realizability topos}, $\rt(\kltwo)$ is
the realizability topos on $\kltwo$. Then $\rt(\kltwo)$ has as a
subcategory, the \emph{Kleene-Vesley topos}, $\kv$, which is defined
as the relative realizability topos
$\rt(\kltwo^\mathrm{rec}, \kltwo)$. See \cite[Section 4.5]{vanoosten}
for more details.

We can embed subspaces of $\realno^n$ into $\rt(\kltwo)$. A subspace
of $\realno^n$ is in particular a countably based $T_0$-space, which
Bauer showed in \cite{bauereqsptte} embed into $\mathbf{PER}(\kltwo)$,
which in turn embeds into $\rt(\kltwo)$. Note however, that for the
special case of subspaces of $\realno^n$, we can more explicitly
describe the embedding into $\asm(\kltwo)$. Given a subspace
$X$ of $\realno^n$, we take the underlying set of the assembly to be
$X$ itself, and we define the existence predicate, $E$, by taking
$E(x)$ to be the set of (functions encoding) Cauchy sequences of
rationals that converge to $x$, for each $x \in X$.

Hence the endpoint inclusion into the topological interval
$\delta_0 \colon 1 \rightarrow [0,1]$, can be viewed as a map in
$\rt(\kltwo)$. Moreover, since the map is evidently computable, it in
fact lies in the subcategory $\kv$.

\begin{definition}
  We say a map in $\kv$ is a \emph{computable Hurewicz fibration} if
  it has the fibred right lifting property against the following
  (trivial) family of maps.
  \begin{equation*}
    \xymatrix{ 1 \ar[rr]^{\delta_0} \ar[dr] & & [0, 1] \ar[dl] \\
      & 1 & }
  \end{equation*}
\end{definition}

Note that since this is the fibred right lifting property, it is
equivalent to having the right lifting property against the map
$\delta_0 \times X \colon X \rightarrow X \times [0,1]$, for every
object $X$ of $\kv$. This justifies the name computable Hurewicz
fibration, by analogy with Hurewicz fibrations in topology.

\begin{theorem}
  There is an awfs on $\kv$ where the maps that admit the structure of
  a right map are precisely the computable Hurewicz fibrations.
\end{theorem}

\begin{proof}
  It suffices to show that $\kv$ is a $\Pi W$-pretopos and satisfies
  $\wisc$. Van den Berg showed in \cite{vdbergpredtop} that this is the case
  for internal realizability toposes, as long as it holds in the
  background. However, Birkendal and Van Oosten showed in
  \cite{birkedalvanoosten} that relative realizability toposes can be
  viewed as internal realizability toposes in $\set^\btwo$, so $\kv$
  is indeed a $\Pi W$-pretopos satisfying $\wisc$. We can now apply
  corollary \ref{cor:pwtoposcofibgen}.
\end{proof}

\subsection{Cubical Assemblies}
\label{sec:cubical-assemblies}

We will construct a category of internal presheaves in $\asm(\klone)$
which we will call the category of \emph{cubical assemblies}, which
will be a realizability variant of the category of cubical sets
defined by Cohen, Coquand, Huber and M\"{o}rtberg in
\cite{coquandcubicaltt}. The definitions of Kan trivial fibration and
fibration are based on the presentation in
\cite[Section 7.5.4]{swanliftprob}\todo{maintain ref to other paper}.

First, note that we can view the free de Morgan algebra on a countable
set $\names$ as follows. We write $\dm_0(\names)$ for the set of
strings in the language of de Morgan algebras with constants from
$\names$. Then $\dm(\names)$ is the quotient of $\dm_0(\names)$ by the
appropriate equalities corresponding the de Morgan algebra axioms. We
write $\phi \equiv \psi$ if $\phi$ and $\psi$ are words that are
identified in $\dm(\names)$. Clearly there is a G\"{o}delnumbering of
$\dm_0(\names)$. Given $\phi \in \dm_0(A)$, we write the corresponding
G\"{o}delnumber as $\gn \phi$.

We define an internal category in assemblies as follows. We take the
underlying small category to be the same as for CCHM cubical
sets. That is, the full subcategory of the Kleisli category on $\dm$
with objects the finite subsets of $\names$. We then need to define
existence predicates $E_0$ and $E_1$ for the objects and
morphisms. Given a finite subset $A$ of $\names$, we define $E_0(A)$
to consist of lists $\langle a_1, \ldots, a_n \rangle$ such that
$A = \{a_1, \ldots, a_n \}$. Given a morphism
$\theta \colon A \rightarrow B$, we define $E_1(\theta)$ to consist of
triples $\langle d, c, e \rangle$, where $d$ and $c$ are codes for the
domain and codomain, and $e$ tracks the function
$A \rightarrow \dm (B)$ underlying $\theta$. That is, given $a \in A$,
$\theta \gn{a}$ is defined and equal to $\gn{\phi}$ for some $\phi$
such that $\phi \equiv \theta(a)$. We call this internal category the
\emph{cube category}.

We now define the category of cubical assemblies to be the category of
diagrams for the cube category. Note that the forgetful
functor $\Gamma \colon \asm(\klone) \rightarrow \set$ extends to a functor
from cubical assemblies to cubical sets.

We define an interval object $\intv$ as the following cubical
assembly. The underlying cubical set is the same as the interval in
CCHM cubical sets. Namely, we take $\intv(A)$ to be $\dm(A)$. We
define the existence predicate on $\intv(A)$ by taking $E([\phi])$ to
be the set consisting of $\gn{\psi}$ for $\psi$ such that
$\psi \equiv \phi$.

We define the face lattice, $F$, to be the quotient of $\intv$ by the
following equivalence relation. We define $[\phi] \sim [\psi]$ when
$\phi \equiv 1 \Leftrightarrow \psi \equiv 1$ holds in cubical
assemblies. In $\asm(\klone)$, this says that for $[\phi], [\psi] \in
F(A)$, $[\phi] \sim [\psi]$ when for every $\theta \colon
A \rightarrow B$ in the cube category, $F(\theta)(\phi) \equiv 1
\Leftrightarrow F(\theta)(\psi) \equiv 1$.

As Coquand et al remark in \cite[Section 3]{coquandcubicaltt}, free de
Morgan algebras have decidable equality. In fact the equality in
$\dm(A)$ is uniformly computably decidable over all finite subsets $A$
of $\names$, and so $\intv$ has decidable equality in
$\asm(\klone)$. 

We will check that the map $\top \colon 1 \rightarrow F$ has decidable
image. Note that since $\asm(\klone)$ does not have effective
quotients in general we need to be a little careful.

Suppose we are given an element of $F(A)$ of the form $[\phi]$. By
decidability of $\equiv$ we know that $\phi \equiv 1$ or
$\phi \not \equiv 1$. In the former case we clearly have $\phi \sim 1$
and so $[\phi] = 1$. We now show that in the latter case
$[\phi] \neq 1$. Just using the fact that the quotient is a
coequalizer and again that $\intv$ has decidable equality we can
define a map $f \colon F(A) \rightarrow 2$ such that $f([\psi]) = 1$
when $\psi \equiv 1$ and $f([\psi]) = 0$ when $\psi \not \equiv
1$. But then $f([\phi]) = 0$ and $f(1) = 1$, so we can deduce $[\phi]
\neq 1$.

In fact one can deduce that this particular quotient is effective, but
we don't need that here.

Therefore, by theorem \ref{thm:simplecofgen} there exists a (strongly
fibred) awfs cofibrantly generated by the following family of maps,
which we refer to as the awfs of Kan cofibrations and trivial
fibrations.
\begin{equation*}
  \begin{gathered}
    \xymatrix{ 1 \ar[rr]^{\top} \ar[dr]
      & & F \ar@{=}[dl] \\
      & F & }
  \end{gathered}  
\end{equation*}

Finally note that the Leibniz product $\delta_0 \hat \times \top$, is
the subobject of $F \times \intv$, which at $A$ consists of
$([\phi], [\psi])$ in $F(A) \times \intv(A)$ such that $\phi \equiv 1$
or $\psi \equiv 0$\footnote{The easiest way to show this is simply to
  verify directly that this definition satisfies the universal
  property of the pushout. In fact one can show that this map is a
  cofibration, but we won't cover this in more detail here.}.  It
follows that $\delta_0 \hat \times \top$ is also locally decidable. It
follows again by theorem \ref{thm:simplecofgen} that there is a
(fibred) awfs cofibrantly generated by the family of maps below, which
we refer to as the awfs of Kan trivial cofibrations and fibrations.
\begin{equation*}
  \begin{gathered}
    \xymatrix{ \intv +_1 F \ar[rr]^{\delta_0 \hat \times \top} \ar[dr]
      & & \intv
      \times F \ar[dl] \\
      & F & }
  \end{gathered}  
\end{equation*}

\section{Conclusion}
\label{sec:conclusion}

\subsection{Comparison With Existing Constructions of Higher Inductive Type}
\label{sec:comp-with-other}

As remarked in the introduction, $W$-types with reductions may be a
special cases of free algebra over varieties (as defined by Blass in
\cite{blassfreealg}), and of QIITs, as developed by Altenkirch,
Capriotti, Dijkstra and Forsberg in \cite{acdfqiit}. We were able to
show initial algebras can be constructed in a wide variety of
categories. For algebraic varieties, Blass observed that initial
algebras can be constructed in any topos with natural number object
satisfying the internal axiom of choice, which is a much smaller class
than the one we considered. However, the construction in section
\ref{sec:constr-init-algebr} is fairly flexible, and may lead to a
refinement of Blass' result, as in the conjecture below.
\begin{conjecture}
  Free algebras for varieties exist in any $\Pi W$-pretopos that
  satisfies $\wisc$.
\end{conjecture}
In fact this has already been conjectured in \cite[Section
8]{vdbergpredtop}, where the question is attributed to Alex Simpson.

The question of when QIITs can be constructed remains open, although
in loc. cit., Altenkirch et al do make some progress towards a
solution. The technique used in section \ref{sec:constr-init-algebr}
might also be helpful here.

In \cite{lumsdaineshulmanhit}, Lumsdaine and Shulman give a very
general approach to the semantics of higher inductive types in
homotopy type theory. Although the set up is quite different, the
problem of constructing the higher inductive types turns out to be
quite similar to the problems we saw in this paper. For this Lumsdaine
and Shulman use some general transfinite constructions due to Kelly
\cite{kellytransfinite}. Unfortunately this approach is not suitable
for the examples we consider here, as we discuss further in the next
section.

\subsection{Other Approaches to the Construction of Initial Algebras}
\label{sec:other-appr-constr}

In section \ref{sec:constr-init-algebr-1} we gave a relatively direct
proof, in place of an application of existing results from
literature. The reader might wonder why this is the case.

A commonly used approach to constructing initial algebras is to use a
transfinite construction. Following Garner's small object argument
\cite{garnersmallobject}, we might try to use one of the general
theorems of Kelly from \cite{kellytransfinite}. However, such
constructions have the disadvantage that they make essential use of
transfinite colimits of ordinal indexed sequences. This means they
will not work for general elementary toposes, which need not be
cocomplete. This is critical here, because our examples are
based on realizability toposes, which are certainly not cocomplete.

It is also difficult to simply carry out a similar transfinite
construction internally in the $\Pi W$-pretopos, since it is unclear
how to formulate ordinals in the internal language in way that
the set theoretic arguments can be easily transferred.

Another possible approach would be to use an internal version of the
special adjoint functor theorem as developed by Day in
\cite{dayadjointfun} or Par{\'e} and Schumacher in
\cite{pareschumacheradjfun}. In fact Par{\'e} and Schumacher indicate
in \cite[Section V.2]{pareschumacheradjfun} how their result can be
used to construct free algebras of certain endofunctors. However, it
is unclear how to show that the pointed endofunctors here satisfy the
necessary conditions to apply the internal special adjoint functor
theorem. Indeed in the paragraph at the end of loc cit. Par{\'e} and
Schumacher remark that the addition of equations makes things more
problematic and suggest using in this case the more powerful results
of Rosebrugh in \cite{rosebrughcoequalizers}. However, Rosebrugh's
proofs apply only to internal toposes of sheaves inside toposes
satisfying the axiom of choice. This again would eliminate our
examples based on realizability. Blass proved in \cite{blassfreealg}
that some form of the axiom of choice really is necessary for
Rosebrugh's result to hold, although like with our results it may be
possible to adapt Rosebrugh's proofs to use a weak form of choice such
as $\wisc$. There is also the issue that the techniques of
Rosebrugh and of Par{\'e} and Schumacher make heavy
use of impredicative notions such as the subobject classifier and the
assumptions of well poweredness and cowell poweredness, and will thus
not apply to $\Pi W$-pretoposes without further work.

\subsection{Directions for Future Work}
\label{sec:direct-future-work}

\subsubsection{Is Choice Really Necessary?}
\label{sec:choice-really-necess}

In our construction of arbitrary $W$-types with reductions in a
$\Pi W$-pretopos we relied on the axiom $\wisc$. It's natural to ask
whether $\wisc$ was really necessary, or whether there's a way to
construct $W$-types with reductions without using any choice.

We saw in section \ref{sec:w-types-with-1} that using classical logic
we can derive all $W$-types with reductions from $W$-types without
using any choice. It might be possible to generalise this result to
all categories of internal presheaves in a boolean topos.

However, we conjecture that in general there are toposes where
some form of choice is strictly necessary, even just for monic
polynomials with reductions.
\begin{conjecture}
  \begin{enumerate}
  \item There is a topos with natural number object with a monic
    polynomial with reductions that does not have an initial algebra.
  \item It is consistent with $\mathbf{IZF}$ that there is a monic
    polynomial with reductions in the category of sets that does not
    have an initial algebra.
  \end{enumerate}
\end{conjecture}
Note that by theorem \ref{thm:rwinitialalg} we know that $\wisc$ has
to fail in the above conjecture. Also by theorem
\ref{thm:initialalgps} we know that the topos cannot be a category of
internal presheaves over a boolean topos (and in particular cannot be
boolean itself).

\subsubsection{Applications to the Semantics of Homotopy Type Theory}
\label{sec:appl-high-induct}

The main aim of this work is towards the semantics of homotopy type
theory and in particular better understanding and generalising the
cubical set model of type theory. We have already seen one aspect of
this, which is that $W$-types with reductions can be used to construct
awfs's where $R$-algebra structures correspond to Kan filling
operators (which in turn are used in the interpretation of dependent
types). We note that in fact we don't need all $W$-types with
reductions in order to do this, but only those where the map
$f \circ k$ in \eqref{eq:27} is a cofibration (assuming cofibrations
are closed under coproduct and pullback). We'll refer to such
polynomials with reductions as \emph{cofibrant}.

Cofibrant $W$-types with reductions may also have further applications
to the semantics of type theory. In \cite{coquandcubicaltt}, Coquand
et al implement higher inductive types by freely adding an
$\mathsf{hcomp}$ operator to a type. This can be seen as a kind of
weak fibrant replacement that can be phrased as a cofibrantly
generated rawfs, as we developed in section
\ref{sec:lift-probl-squar}. An important point is that this
construction is stable under pullback, which corresponds to our notion
of strongly fibred rawfs.  We again notice that we
only need cofibrant $W$-types with reductions.

The author hopes to develop these ideas further in a future paper. The
following conjecture illustrates the kind of result expected.
\begin{conjecture}
  Let $\catc$ be a topos with natural number object. Suppose further
  that $\catc$ satisfies all of the axioms considered by Orton and
  Pitts in \cite{pittsortoncubtopos}. Suppose further that initial
  algebras exist for all cofibrant polynomials with reductions. Then
  pushouts, $n$-truncations, set-quotients, suspensions and
  $n$-spheres can be implemented in the resulting CwF.
\end{conjecture}

\subsubsection{Algebraic Model Structures on Realizability Toposes}
\label{sec:algebr-model-struct}

In section \ref{sec:exampl-prev-unkn} we saw three examples of awfs's
based on realizability. It's natural to ask whether these in fact form
part of algebraic model structures (as defined by Riehl in
\cite{riehlams}). We conjecture that in fact this is possible.

Firstly, by generalising results by Sattler in
\cite{sattlermodelstructures} the author expects it will be possible
to prove the following conjectures.
\begin{conjecture}
  The two awfs's in section \ref{sec:effective-topos} form part of an
  algebraic model structure on the effective topos.
\end{conjecture}
\begin{conjecture}
  The two awfs's in section \ref{sec:cubical-assemblies} form part of an
  algebraic model structure on the category of cubical assemblies.
\end{conjecture}

The status of the example in $\kv$ is less clear, but by analogy with
the well known model structure on topological spaces by Str{\o}m
\cite{stromhcishc}, the following conjecture might also be true.
\begin{conjecture}
  The awfs in section \ref{sec:funct-real} forms the trivial
  cofibrations and fibrations part of an algebraic model structure on
  the Kleene-Vesley topos.
\end{conjecture}

\subsection*{Acknowledgements}
\label{sec:acknowledgements}

I'm grateful to Benno van den Berg for many helpful discussions and
suggestions while developing this work.

\bibliographystyle{abbrv}
\bibliography{mybib}{}

\begin{thebibliography}{10}

\bibitem{acdfqiit}
T.~Altenkirch, P.~Capriotti, G.~Dijkstra, and F.~N. Forsberg.
\newblock Quotient inductive-inductive types.
\newblock arXiv:1612.02346, 11 2016.

\bibitem{altenkirchkaposiqit}
T.~Altenkirch and A.~Kaposi.
\newblock Type theory in type theory using quotient inductive types.
\newblock In {\em Proceedings of the 43rd Annual ACM SIGPLAN-SIGACT Symposium
  on Principles of Programming Languages}, POPL '16, pages 18--29, New York,
  NY, USA, 2016. ACM.

\bibitem{bauereqsptte}
A.~Bauer.
\newblock A relationship between equilogical spaces and type two effectivity.
\newblock {\em Mathematical Logic Quarterly}, 48(S1):1--15, 2002.

\bibitem{birkedalvanoosten}
L.~Birkedal and J.~van Oosten.
\newblock Relative and modified relative realizability.
\newblock {\em Annals of Pure and Applied Logic}, 118(1):115 -- 132, 2002.

\bibitem{blassfreealg}
A.~Blass.
\newblock Words, free algebras, and coequalizers.
\newblock {\em Fundamenta Mathematicae}, 117(2):117--160, 1983.

\bibitem{carbonilackwalters}
A.~Carboni, S.~Lack, and R.~Walters.
\newblock Introduction to extensive and distributive categories.
\newblock {\em Journal of Pure and Applied Algebra}, 84(2):145 -- 158, 1993.

\bibitem{coquandcubicaltt}
C.~Cohen, T.~Coquand, S.~Huber, and A.~M\"{o}rtberg.
\newblock Cubical type theory: a constructive interpretation of the \
  univalence axiom.
\newblock arXiv:1611.02108, 2015.

\bibitem{dayadjointfun}
B.~Day.
\newblock An adjoint-functor theorem over topoi.
\newblock {\em Bulletin of the Australian Mathematical Society},
  15(3):381--394, 1976.

\bibitem{gambinohylanddepw}
N.~Gambino and M.~Hyland.
\newblock Wellfounded trees and dependent polynomial functors.
\newblock In S.~Berardi, M.~Coppo, and F.~Damiani, editors, {\em Types for
  Proofs and Programs: International Workshop, TYPES 2003, Torino, Italy, April
  30 - May 4, 2003, Revised Selected Papers}, pages 210--225. Springer Berlin
  Heidelberg, Berlin, Heidelberg, 2004.

\bibitem{gambinosattlerpi}
N.~Gambino and C.~Sattler.
\newblock The frobenius condition, right properness, and uniform fibrations.
\newblock {\em Journal of Pure and Applied Algebra}, 221(12):3027 -- 3068,
  2017.

\bibitem{garnersmallobject}
R.~Garner.
\newblock Understanding the small object argument.
\newblock {\em Applied Categorical Structures}, 17(3):247--285, 2009.

\bibitem{hofmannlcc}
M.~Hofmann.
\newblock On the interpretation of type theory in locally cartesian closed
  categories.
\newblock In {\em Computer Science Logic: 8th Workshop, CSL '94, Kazimierz,
  Poland}, number 933 in Lecture Notes in Computer Science. Springer, 1994.

\bibitem{jacobs}
B.~Jacobs.
\newblock {\em Categorical Logic and Type Theory}.
\newblock Number 141 in Studies in Logic and the Foundations of Mathematics.
  North Holland, Amsterdam, 1999.

\bibitem{kellytransfinite}
G.~Kelly.
\newblock A unified treatment of transfinite constructions for free algebras,
  free monoids, colimits, associated sheaves, and so on.
\newblock {\em Bulletin of the Australian Mathematical Society}, 22(1):1–83,
  1980.

\bibitem{lumsdaineshulmanhit}
P.~L. Lumsdaine and M.~Shulman.
\newblock Semantics of higher inductive types.
\newblock arXiv:1705.07088, May 2017.

\bibitem{maiettimodular}
M.~E. Maietti.
\newblock Modular correspondence between dependent type theories and categories
  including pretopoi and topoi.
\newblock {\em Mathematical Structures in Computer Science}, 15:1089--1149, 12
  2005.

\bibitem{moerdijkpalmgrenwtypes}
I.~Moerdijk and E.~Palmgren.
\newblock Wellfounded trees in categories.
\newblock {\em Annals of Pure and Applied Logic}, 104(1):189 -- 218, 2000.

\bibitem{moerdijkpalmgrenast1}
I.~Moerdijk and E.~Palmgren.
\newblock Type theories, toposes and constructive set theory: predicative
  aspects of ast.
\newblock {\em Annals of Pure and Applied Logic}, 114(1):155 -- 201, 2002.
\newblock Troelstra Festschrift.

\bibitem{pittsortoncubtopos}
I.~Orton and A.~M. Pitts.
\newblock Axioms for modelling cubical type theory in a topos.
\newblock In J.-M. Talbot and L.~Regnier, editors, {\em 25th EACSL Annual
  Conference on Computer Science Logic ({CSL} 2016)}, volume~62 of {\em Leibniz
  International Proceedings in Informatics (LIPIcs)}, pages 24:1--24:19,
  Dagstuhl, Germany, 2016. Schloss Dagstuhl--Leibniz-Zentrum f\"ur Informatik.

\bibitem{pareschumacheradjfun}
R.~Par{\'e} and D.~Schumacher.
\newblock Abstract families and the adjoint functor theorems.
\newblock In {\em Indexed Categories and Their Applications}, pages 1--125.
  Springer Berlin Heidelberg, Berlin, Heidelberg, 1978.

\bibitem{Rathjen94}
M.~Rathjen.
\newblock The strength of some {M}artin-{L}\"{o}f type theories.
\newblock {\em Archive For Mathematical Logic}, 33:347--385, 1994.

\bibitem{riehlams}
E.~Riehl.
\newblock Algebraic model structures.
\newblock {\em New York Journal of Mathematics}, 17:173--231, 2011.

\bibitem{robertswisc}
D.~M. Roberts.
\newblock The weak choice principle wisc may fail in the category of sets.
\newblock {\em Studia Logica}, 103(5):1005--1017, Oct 2015.

\bibitem{rosebrughcoequalizers}
R.~Rosebrugh.
\newblock Coequalizers in algebras for an internal type.
\newblock In {\em Indexed Categories and Their Applications}, pages 243--260.
  Springer Berlin Heidelberg, Berlin, Heidelberg, 1978.

\bibitem{sattlermodelstructures}
C.~Sattler.
\newblock The equivalence extension property and model structures.
\newblock arXiv:1704.06911, 2017.

\bibitem{seelylcctt}
R.~A.~G. Seely.
\newblock Locally cartesian closed categories and type theory.
\newblock {\em Mathematical Proceedings of the Cambridge Philosophical
  Society}, 95(1):33–48, 1984.

\bibitem{stromhcishc}
A.~Str{\o}m.
\newblock The homotopy category is a homotopy category.
\newblock {\em Archiv der Mathematik}, 23(1):435--441, Dec 1972.

\bibitem{swanliftprob}
A.~W. Swan.
\newblock Lifting problems in {G}rothendieck fibrations.
\newblock arXiv:1802.06718, February 2018.

\bibitem{hottbook}
{Univalent Foundations Program}.
\newblock {\em Homotopy Type Theory: Univalent Foundations of Mathematics}.
\newblock \url{http://homotopytypetheory.org/book}, Institute for Advanced
  Study, 2013.

\bibitem{vdbergpredtop}
B.~van~den Berg.
\newblock Predicative toposes.
\newblock arXiv:1207.0959v1, 2012.

\bibitem{vdbergfrumin}
B.~van~den Berg and D.~Frumin.
\newblock A homotopy-theoretic model of function extensionality in the
  effective topos.
\newblock arXiv:1701.08369, January 2017.

\bibitem{vdbergdemarchi}
B.~van~den Berg and F.~D. Marchi.
\newblock Non-well-founded trees in categories.
\newblock {\em Annals of Pure and Applied Logic}, 146(1):40 -- 59, 2007.

\bibitem{vanoosten}
J.~van Oosten.
\newblock {\em Realizability: An Introduction to its Categorical Side}, volume
  152 of {\em Studies in Logic and the Foundations of Mathematics}.
\newblock Elsevier, 2008.

\end{thebibliography}

\end{document}